%% file: FGS_DistortedBM_ARXIVSecVersion.tex
\documentclass[10 pt]{report}
\usepackage[disable]{todonotes}
\usepackage{todonotes}
\usepackage{amsmath, amssymb, amsthm, mathtools, mathrsfs}	 
\usepackage{bbm}
\usepackage{dsfont} 
\usepackage{nicefrac}
\usepackage{siunitx} 
\usepackage{geometry}
\usepackage{xspace}
\usepackage[shortlabels]{enumitem} 
\usepackage{titling}
\usepackage{url}
\urldef{\urluni}{\url}{https://math.rptu.de/ags/fuana}
\urldef{\emailfattler}{\url}{torben.fattler@math.rptu.de}
\urldef{\emailgrothaus}{\url}{grothaus@rptu.de}
\urldef{\emailsteil}{\url}{nathalie.steil@rptu.de}
\usepackage[symbol]{footmisc}
\usepackage[style=alphabetic,
citestyle=alphabetic,
maxcitenames=3,
maxbibnames=4,
uniquename=false,
uniquelist=false,
eprint=true]{biblatex}  
\usepackage{biblatex}  

\usepackage[hypertexnames=false]{hyperref}

\input{macros/GeneralMathMacros}
\input{macros/OtherMacros}	
\input{macros/StochasticMacros}  

\begin{document}
 \makeatletter
 \begin{titlepage}
    \begin{center}
    \textbf{\Large Construction of distorted Brownian motion with permeable sticky behaviour on sets with Lebesgue measure zero}\\
    \vspace{1 cm}
     Torben Fattler\footnote[1]{University of Kaiserslautern-Landau, 67663 Kaiserslautern, Germany.}\footnote[2]{\urluni}\footnote[3]{\emailfattler}, Martin Grothaus\footnotemark[1]\footnotemark[2]\footnote[4]{\emailgrothaus}, Nathalie Steil\footnotemark[1]\footnotemark[2]\footnote[6]{\emailsteil}
    \end{center}
    \vspace{1 cm}
    \vspace{1 cm}
    
    {\underline{Abstract:} 
    The starting point is a gradient Dirichlet form with respect to $\measdens{\lambda^d}{\varrho}$ on the space  $L^2({\mathbb R}^d, \measdens{\mu}{\varrho})$. Here $\lambda^d$ is the Lebesgue measure on ${\mathbb R}^d$, $\varrho$ a strictly positive density and $\mu$ puts weight on a set $\SetA \subset {\mathbb R}^d$ with Lebesgue measure zero. We show that the Dirichlet form admits an associated stochastic process $X$. 
    We derive an explicit representation of the corresponding generator if $\SetA$ is a Lipschitz boundary. This representation together with the Fukushima decomposition identifies $X$ as a distorted Brownian motion with drift given by the logarithmic derivative of $\varrho$ in ${\mathbb R}^d \setminus A$. Furthermore, we prove $X$ to be irreducible and recurrent. Finally, via ergodicity we show positive séjour time of $X$ on $\SetA$. Hence we obtain a stochastic process $X$ with permeable sticky behaviour on $\SetA$.}
    \par
    {\underline{Acknowledgement:}} M.G.\xspace and N.S.\xspace thank Vitalii Konarovskyi for the stimulating discussion about the relation of the process we construct and the one constructed in \cite{EP2014}. N.S.\xspace also thanks Andreas Gathmann for his useful hints on topological properties of Lipschitz domains.
 \end{titlepage}
     \makeatother
    
	\input{sections/SectionIntroduction_DBM_ARXIV}
	\input{sections/SectionPreliminariesAndNotation_DBM_ARXIV}
	\input{sections/SectionDirchletForms_DBM_ARXIV}
	\input{sections/SectionLipschitzDomainGenerator_DBM_ARXIV}
	\input{sections/SectionProcess_DBM_ARXIV}

\printbibliography
\end{document}

%% file: macros/GeneralMathMacros.tex
\DeclareMathOperator{\R}{\mathbb{R}}

\DeclareMathOperator{\N}{\mathbb{N}}
\DeclareMathOperator{\Z}{\mathbb{Z}}

\newcommand{\myint}[4]{\int\limits_{#1}^{#4}#2\textup{d}#3}

\newcommand{\SetA}{A}

\newcommand{\ContDiff}[3]{\ensuremath{\mathcal{C}^{#2}_{\textup{#3}}(#1)}}

\newcommand{\cupdot}{\mathbin{\mathaccent\cdot\cup}}

\newcommand{\Erho}{\mathcal{E}^{\varrho}}
\newcommand{\ErhoOne}{\mathcal{E}^{\mathbbm{1}}}

\newcommand{\supp}[1]{\textup{supp}\ensuremath{\left(#1\right)}}

\newcommand{\eps}{\ensuremath{\varepsilon}}

\newcommand{\Lipschitz}{\ensuremath{U}}
\newcommand{\LipschitzComplement}{\ensuremath{\overline{U}^{\text{C}}}}
\newcommand{\vdimlower}{\ensuremath{v_{\textbf{.}d-1}}}
\newcommand{\Additive}{\ensuremath{C}}

\newcommand{\extension}[2]{\ensuremath{{#1}_{#2}}}

%% file: macros/OtherMacros.tex
\newtheorem{thm}{Theorem}[section]
\newtheorem{prop}[thm]{Proposition}
\newtheorem{cor}[thm]{Corollary}
\newtheorem{lem}[thm]{Lemma}
\newtheorem{cond}[thm]{Condition}

\theoremstyle{definition}
\newtheorem{definition}[thm]{Definition}

\theoremstyle{remark} 
\newtheorem{rem}[thm]{Remark}

\newcommand{\ie}{i.e.\xspace~}

\newcommand{\almosteverywhere}{a.e.\xspace~}
\newcommand{\eg}{e.g.\xspace~}

\newcommand{\resp}{resp.\xspace~}
\newcommand{\withoutloss}{wlog\xspace~}
\newcommand{\Wlog}{Wlog\xspace~}
\newcommand{\withrespectto}{w.r.t.~}
\newcommand{\esssupp}[1]{\text{ess supp}(\ensuremath{#1})}
\newcommand{\qe}{q.e.\xspace~}

%% file: macros/StochasticMacros.tex
\newcommand{\Rmeasurespace}[1]{(\mathbb{R}^#1, \mathscr{B}(\mathbb{R}^#1), \lambda^#1)}
\newcommand{\Prop}[2][]{\mathbb{P} \ifx\\#1\\\else_{#1}\fi\!\left( #2\right)}
\newcommand{\PropPalm}[2][]{\mathbb{P} \ifx\\#1\\\else^{#1}\fi\!\left( #2\right)}
\newcommand{\Borel}[1]{\mathscr{B}(#1)}

\newcommand{\mS}{\mathcal{S}}

\newcommand{\Markov}[1]{\ensuremath{\textbf{M}^{#1}}}

\newcommand{\Expec}[3]{\mathbb{E}_{#1}^{#2}\left[#3\right]}
\newcommand{\cE}[3][]{\mathbb{E}\ifx\\#1\\\else_{#1}\fi\!\left[#2 \big| {#3}\right]}
\newcommand{\Var}[2][]{\text{Var} \ifx\\#1\\\else_{#1}\fi\!\left(#2\right) }
\newcommand{\Cov}[2][]{\text{Cov} \ifx\\#1\\\else_{#1}\fi\!\left(#2\right) }

\newcommand{\murho}{\ensuremath{\varrho \mu}}
\newcommand{\measdens}[2]{\ensuremath{#2 #1}}

\newcommand{\eucscalar}[2]{\left(#1,#2\right)_{\textup{euc}}}
\newcommand{\eucnorm}[1]{\left\Vert\ensuremath{#1}\right\Vert_{\textup{euc}}}
\newcommand{\Ltwomurho}{\ensuremath{L^2\left(\mathbb{R}^d, \murho\right)}}
\newcommand{\LtwomurhoOneDim}{\ensuremath{L^2\left(\mathbb{R}, \murho\right)}}
\newcommand{\LtwomuOneDim}{\ensuremath{L^2\left(\mathbb{R}, \mu\right)}}

\newcommand{\Ltwomurhoscalar}[2]{\left( #1 , #2 \right)_{\Ltwomurho}}
\newcommand{\Lcaltwo}[2]{\mathscr{L}^2\left(#1, #2\right)}
\newcommand{\Lcaltwomurho}{\mathscr{L}^2\left(\mathbb{R}^d, \murho\right)}
\newcommand{\LcaltwomurhoOneDim}{\mathscr{L}^2\left(\mathbb{R}, \murho\right)}
\newcommand{\Lcaltwoloc}{\mathscr{L}^2_{\text{loc}}\left(\mathbb{R}^d, \lambda^d\right)}
\newcommand{\Ltwoloc}{L^2_{\text{loc}}\left(\mathbb{R}^d, \lambda^d\right)}
\newcommand{\abs}[1]{\left\lvert\ensuremath{#1}\right\rvert}
\newcommand{\norm}[1]{\left\Vert\ensuremath{#1}\right\Vert}
\newcommand{\Laplace}{\ensuremath{\Delta}}
\newcommand{\localSobolev}[2]{H^{#1}_{\textup{loc}}(#2)}
\newcommand{\Sobolev}[2]{H^{#1}(#2)}
\newcommand{\restrictfunc}[2]{\ensuremath{#1 \big| _{#2}}}
\newcommand{\Capac}[2]{\ensuremath{\text{Cap}_{#1}(#2)}}
\newcommand{\Smoothmeasures}{\ensuremath{\mathcal{Q}}}

%% file: sections/SectionIntroduction_DBM_ARXIV.tex
\section{Introduction}
In \cite{B2014} the author studies existence and uniqueness of a weak solution to the equation
\begin{align*}
	dX_t = \mathbbm{1}_{\{X_t\not=0\}}\, dB_t,\quad X_0=0,
\end{align*}
where $\big(B_t\big)_{t\ge 0}$ is a one-dimensional standard Brownian motion, which has the sum of the Lebesgue measure and the Dirac measure at zero as speed measure. At the same time, in \cite{EP2014} the authors proved for the system of equations 
\begin{align*}\left\{
	\begin{array}{ll}
		dX_t &= \mathbbm{1}_{\{X_t\not=0\}}\, dB_t\\
		\mathbbm{1}_{\{X_t=0\}}\,dt &= \frac{1}{\mu}\, d\ell_t^0(X)
	\end{array}\right.,
\end{align*}
where $\mu\in(0,+\infty)$, $X:=\big(X_t\big)_{t\ge 0}$ and $\ell_t^0(X)$ is the local time in the semimartingale sense at $0$ of $X$, existence and uniqueness of a weak solution and non-existence of a strong solution.

In the present paper we provide in dimension $d\in\mathbb{N}$ a weak solution to the system of equations
\begin{align}\left\{
	\begin{array}{ll}
		dX^i_t &= \mathbbm{1}_{\mathbb{R}^d\setminus A}\big(X_t\big)\sqrt{2} \, dB^i_t + \mathbbm{1}_{\mathbb{R}^d\setminus A}\big(X_t\big)\partial_{x_i}\ln(\varrho)(X_t)\,dt\\
		X^i_0 &= x_i
	\end{array}\right.\text{for}\quad i\in\{1,\ldots,d\},
\end{align}\label{eq: sde_sys}
for 
\begin{enumerate}[(i)]
	\item quasi every starting point $x=(x_1,\ldots,x_d)^T\in\mathbb{R}^d$ if $d \geq 2$ and the function $\varrho:\mathbb{R}^d\to \R$ is continuously differentiable, strictly positive and integrable with respect to the measure $\mu:=\lambda ^d+\mS$ on $\big(\R^d,\Borel{\R^d}\big)$,
	\item for all starting points $x \in \R$ if the function $\varrho$ fulfils the same Condition as in (a) or $\varrho$ is continuously differentiable, strictly positive and bounded
\end{enumerate}
with respect to an underlying Dirichlet form. Here $\big(B_t^i\big)_{t\ge 0}$ are independent one-dimensional standard Brownian motions for $i\in\{1,\ldots, d\}$, $\lambda^d$ denotes Lebesgue measure on the Borel-$\sigma$-algebra $\Borel{\R^d}$ of $\mathbb{R}^d$ and $\mS$ is a $\sigma$-finite measure on the trace $\sigma$-algebra $\Borel{\SetA}$ of $\Borel{\R^d}$ over $\SetA$, where $\SetA$ is a closed $\lambda ^d$ null set in $\Borel{\R^d}$, trivially extended to ${\R^d}$. Moreover, the constructed solution spends a positive amount of time in $\SetA$. I.e., we construct a distorted Brownian motion $\big(X_t\big)_{t\ge 0}:=\big((X_t^1,\ldots, X_t^d)\big)_{t\ge 0}$ in $\mathbb{R}^d\setminus \SetA$ with a sticky behaviour in a Lebesgue null set $A$ as a weak solution to the system \eqref{eq: sde_sys} of stochastic differential equations. 
The application of Dirichlet form techniques in the present setting is motivated by the studies in \cite{FGV2016}, where the authors used Dirichlet forms in order to construct reflected distorted Brownian motions with sticky boundary behaviour.

The stochastic processes we construct in this paper are associated with symmetric regular strongly local Dirichlet forms $\big(\Erho, D(\Erho)\big)$ which are the closure of symmetric bilinear forms
\begin{align*}
	\Erho\left(f,g\right) \vcentcolon = \myint{\R^d}{\eucscalar{\nabla f}{\nabla g}\,\varrho}{\lambda^d}{},\quad f,g\in D:=\ContDiff{\R^d}{\infty}{c}\subseteq \Ltwomurho,
\end{align*}
in $\Ltwomurho$.
Here, $\ContDiff{\R^d}{\infty}{c}$ denotes the space of infinitely often differentiable functions on $\mathbb{R}^d$ with compact support and $\eucscalar{\cdot}{\cdot}$ the euclidean scalar product on $\mathbb{R}^d$. 
For a specific choice of a Lebesgue null set $A$, we obtain an explicit representation of its generator. This reads, for functions $f$ that are continuous with compact support on $\R^d$ and twice continuously differentiable on $\R^d\setminus\SetA$ with bounded derivatives,
\begin{align*}
	L^{\varrho}f = \mathbbm{1}_{\R\setminus \SetA}\left(f'' + f' (\ln(\varrho))'\right) + \mathbbm{1}_{\SetA}\left(\extension{(f')}{r}-\extension{(f')}{l}\right)
\end{align*}
in dimension $d=1$ under Condition \ref{cond: d is one}, e.g.~$A:=\{0\}$ and
\begin{align*}
	L^{\varrho}f = \mathbbm{1}_{\R^d\setminus \partial U}\left(\Laplace f + \eucscalar{\nabla f}{\nabla \ln(\varrho)}\right) + \mathbbm{1}_{\partial U}\left(\extension{\left(\partial_{\nu_{\Lipschitz}}f\right)}{\LipschitzComplement}-\extension{\left(\partial_{\nu_{\Lipschitz}}f\right)}{\Lipschitz}\right)
\end{align*}
in dimension $d\ge 2$ under Condition \ref{cond: Lipschitz boundary} for $A:=\partial U$, where $U$ is open and bounded with Lipschitz boundary.

Having a strongly local regular Dirichlet form in hand, abstract theory provides the existence of a diffusion process $\mathbf{M}^\varrho$ with state space $\mathbb{R}^d$ associated to $\Erho$. Via a Fukushima decomposition of the constructed process $\mathbf{M}^{\varrho}$, we are able to show that this process is solving \eqref{eq: sde_sys} for all starting points if $d=1$ and for quasi every starting point if $d \geq 2$ with respect to the associated Dirichlet form $\Erho$ for a choice of $\SetA$ as in the derivation of the generator. In particular, we obtain for quasi all $x=(x_1,\ldots,x_d)^T\in\mathbb{R}^d$ if $d \geq 2$ and for all $x \in \R$ if $d=1$, the following representation for the coordinate process $X:=\big(X^1_t, \ldots, X_t^d\big)_{t\ge 0}$ of the diffusion process $\mathbf{M}^{\varrho}$:
\begin{align*}
	X_t^i = X_0^i + \sqrt{2} \myint{0}{\mathbbm{1}_{\R^d\setminus\SetA}(X_s)\,}{B^i_s}{t} + \myint{0}{\mathbbm{1}_{\R^d\setminus\SetA}(X_s)\partial_{x_i}(\ln \varrho(X_s))\,}{s}{t}, \quad X_0^i=x_i
\end{align*}
for all $t \geq 0$ and for all $1 \leq i\leq d$.
That the behaviour of the constructed process in $A$ indeed is sticky, we conclude by proving ergodicity of $\mathbf{M}^\varrho$. A key ingredient here is, that the underlying Dirichlet form turns out to be irreducible. Therefore, we are able to show that the séjour time on $A$ is positive. On the other hand, this implies along with the recurrence property of the underlying Dirichlet form that $\SetA$ is permeable for the process.

Our paper is organized as follows.
In Section $2$, we provide the basic setting and notations. In Section $3$, we introduce the bilinear form \eqref{eq: Bilinearform} and prove in Proposition \ref{prop: Diri from regular and strongly local} that its closure $(\Erho, D(\Erho))$ is a regular and strongly local Dirichlet form. If the density $\varrho$ is integrable \withrespectto $\lambda^d$ or $d=1$ and $\varrho$ is bounded, Proposition \ref{prop: E is recurrent} shows that  $(\Erho, D(\Erho))$ is recurrent and hence, conservative.

In Section $4$, for special choices of $\SetA$, we give an explicit representation of the corresponding generator for functions that are continuous with compact support on $\R^d$, twice continuously differentiable on $\R^d\setminus\SetA$ with bounded derivatives, see Theorems \ref{thm: Calculation Generator d is one} in the one-dimensional case and \ref{thm: Calculation Generator d greaterequal two} for $d \geq 2$.

In Section $5$, we show in Theorem \ref{thm: Existence of processs} that $(\Erho, D(\Erho))$ admits an associated diffusion process, \ie a strong Markov process with continuous sample path and, if $(\Erho, D(\Erho))$ is conservative, with infinite lifetime. 
In the Theorems \ref{thm: Fuk decomp in 1d} for $d=1$ and \ref{thm: Fukushima decomp} for $d \geq 2$, we provide the Fukushima decomposition of the process. This is a key ingredient to show that the process solves  the stochastic differential equation in \eqref{eq: sde_sys} and identifies the process as a distorted Brownian motion, see Theorems \ref{thm: Process is Bm} for $d \geq 2$ and \ref{thm: Process is Bm in 1d} for $d=1$. Finally, we prove that $(\Erho, D(\Erho))$ is irreducible under the Condition that $\murho$ is finite, see Proposition \ref{prop: E is irreducible}. Together with recurrence, this implies that the process is ergodic, see Theorem \ref{thm: process is ergodic}. As direct consequences, the séjour time of the process on $\SetA$ is positive, see Corollary \ref{cor: positive sejour time}, and the set $\SetA$ is permeable for the process, see Corollary \ref{cor: permeable}.
In Remark \ref{rem: Bass process compared to our process}, we show that in the case $d = 1$, $\varrho = \mathbbm{1}$ and $A =\{0\}$ the process we construct coincides with the one constructed in \cite{B2014}.
Furthermore, in the latter reference, see \cite{B2014}, Remark 5.3 (p.12), there is a sketch of a proof identifying the above process with the process constructed in \cite{EP2014}. In particular, this shows that our process coincides with a process constructed by a time change of a Brownian motion.

The following list of main results summarises the progress achieved in this paper:
\begin{enumerate}[(i)]
	\item We construct a diffusion process with infinite lifetime and sticky behaviour on a specified set $\SetA$.
	\item We explicitly derive a representation of its generator.
	\item Using the structure of the generator and the Fukushima decomposition we show that the constructed process solves the underlying stochastic differential equation for quasi all starting points if $d \geq 2$ and for all starting points if $d=1$.
	\item We prove that the set $\SetA$ is permeable and, via an ergodicity result, we show that the constructed process has sticky behaviour on $\SetA$.
\end{enumerate}

%% file: sections/SectionPreliminariesAndNotation_DBM_ARXIV.tex
\section{Preliminaries and Notation}
In the whole paper, we endow $\R^d$ with its standard topology and subsets $B \subseteq \R^d$ with the corresponding trace topology. We denote the euclidean norm on $\R^d$ by $\eucnorm{\cdot}$, the uniform norm on $\R^d$ by $\norm{\cdot}_{\text{sup}}$ and the absolute value on $\R$ by $\abs{\cdot}$. For any set $B$ in the Borel $\sigma$-algebra on $\R^d$ denoted by $\Borel{\R^d}$, the indicator function of $B$ is defined by $\mathbbm{1}_B \vcentcolon \R^d \to \R,\ x \mapsto \begin{cases} 1, &\text{if }x \in B\\ 0, &\text{if }x \notin B \end{cases}$.\\ 
For a function $f$ defined on $B$, we denote by $\mathbbm{1}_Bf$ the trivial extension of $f$ to $\R^d$.  If $f$ is defined on $\R^d$, then $\mathbbm{1}_Bf$ denotes the function that is equal to $f$ on $B$ and zero on $\R^d\setminus B$. For an arbitrary subset $F \subseteq \R^d$ we denote the Borel $\sigma$-algebra on $F$ by $\Borel{F}$. 

For a fixed $d \in \N$ we consider the measure space $\Rmeasurespace{d}$, where $\lambda^d$ denotes Lebesgue measure.
Further, for a closed $\lambda^d$ null set $\SetA$, we consider $\left(\SetA, \Borel{\SetA}, \mS\right)$, where $\mS$ denotes a $\sigma$--finite measure on $\Borel{\SetA}$.  We set $\mu \vcentcolon \Borel{\R^d} \to [0, \infty], F  \mapsto  \lambda^d(F)+ \mS(\SetA \cap F).$
For some continuously differentiable $\varrho \vcentcolon \R^d \to \R$ with $\varrho>0$ we define the measure
$\murho$ on $(\R^d, \Borel{\R^d})$ by $\murho(F) = \myint{\R^d}{\mathbbm{1}_F\, \varrho}{\mu}{}$
and obtain the real Hilbert space $\Ltwomurho$. We always use this notation for measures with density, \eg $\measdens{\lambda^d}{\varrho}$.

Let $\mathscr{F}(\R^d)$ be a function set that embeds into $\Ltwomurho$. We denote the corresponding embedded set
by $\mathscr{F}(\R^d) \subseteq \Ltwomurho$ and for any $f \in \mathscr{F}(\R^d) \subseteq \Ltwomurho$, we always choose the $\murho$-version in $\mathscr{F}(\R^d)$.
For the particular choice $ \mathscr{F}(\R^d) \vcentcolon =\ContDiff{\R^d}{\infty}{c}$, we set $D\vcentcolon= \ContDiff{\R^d}{\infty}{c} \subseteq \Ltwomurho$. Here $\ContDiff{\R^d}{\infty}{c}$ denotes the space of infinitely often differentiable functions on $\R^d$ with compact support. For $m \in \N$, we denote the space of $m$ times continuously differentiable functions on $\R^d$ (with compact support) by $\ContDiff{\R^d}{m}{(c)}$.

\begin{rem}\label{rem: Cinfcomp is dense in Ltwomurho}
	Since $\R^d$ is a locally compact space and countable at infinity, \cite{BauerEnglisch}, Corollary 7.5.5 (p.220) states that  $\murho$ is regular and that $\ContDiff{\R^d}{0}{c}\subseteq \Ltwomurho$ is dense in $\Ltwomurho$. In particular, $\murho$ is a Radon measure.
	Furthermore, because $\murho(K) < \infty$ for all $K \subseteq \R^d$ compact, $\murho$ is also a Baire measure and therefore, $D$ is dense in $\Ltwomurho$.
\end{rem}

%% file: sections/SectionDirchletForms_DBM_ARXIV.tex
\section{The Dirchlet form \texorpdfstring{$(\Erho, D(\Erho))$}{Erho}}
We define the bilinear form $(\Erho, D)$ via 
\begin{align}\label{eq: Bilinearform}
	\Erho \vcentcolon D \times D \to \R, 
	(f,g) \mapsto \Erho\left(f,g\right) \vcentcolon = \myint{\R^d}{\eucscalar{\nabla f}{\nabla g}\,\varrho}{\lambda^d}{}
	= \sum_{i=1}^d \myint{\R^d}{\partial_{x_i}f\partial_{x_i}g\, \varrho}{\lambda^d}{},
\end{align}
where  $\nabla f$ and $\nabla g$ denote the gradient of $f$ and $g$, respectively, $\partial_{x_i}f$ and $\partial_{x_i}g$ denote the $i$-th partial derivative of $f$ and $g$ for $i\in\{1, \ldots, d\}$, respectively, and $\eucscalar{\cdot}{\cdot}$ denotes the euclidean scalar product on $\R^d$.
\begin{rem}
	The pair $(\Erho, D)$ is a symmetric, non--negative definite bilinear form and by Remark \ref{rem: Cinfcomp is dense in Ltwomurho} densely defined on $\Ltwomurho$.
\end{rem}
\begin{prop}
	The bilinear form $(\Erho, D)$ is closable on $\Ltwomurho$. We denote its closure by $(\Erho, D(\Erho))$.
\end{prop}

\begin{proof}
	We define a linear operator by 
	\begin{align*}
		L \vcentcolon D \to \Ltwomurho, 
		f \mapsto Lf \vcentcolon= \mathbbm{1}_{\R^d\setminus \SetA}\left(\Laplace f+\eucscalar{\nabla f}{\nabla \ln \varrho }\right),
	\end{align*}
	where $\Laplace f$ is the Laplace operator applied to $f$.
	Using integration by parts we have for $f,g \in D$
	\begin{align*}
		&\Ltwomurhoscalar{-Lf}{g}
		= \myint{\R^d \setminus \SetA}{-\left(\Laplace f+\eucscalar{\nabla f}{\nabla \ln \varrho }\right)g\,\varrho}{\lambda^d}{}
		=  \myint{\R^d}{\eucscalar{\nabla f}{\nabla g}\,\varrho}{\lambda^d}{}\\+&\myint{\R^d}{\eucscalar{\nabla f}{\nabla \varrho}g\,}{\lambda^d}{}
		-  \myint{\R^d}{\eucscalar{\nabla f}{\nabla \varrho}g\,}{\lambda^d}{}=
		\myint{\R^d}{\eucscalar{\nabla f}{\nabla g} \,\varrho}{\lambda^d}{}=
		\Erho\left(f,g\right).
	\end{align*}
	Since the bilinear form $\left(\Erho, D\right)$ is symmetric and non--negative definite,  $(L,D)$ is symmetric and non--positive definite and  \cite{MaRoeckner},  Proposition I. 3.3 (p.29) yields the claim.
\end{proof}
For later reference, we state:
\begin{equation}\label{defn: Erho}
	\parbox{0.8\textwidth}{$(\Erho, D(\Erho))$ is the closure of the bilinear form given in \eqref{eq: Bilinearform}, where $D = \ContDiff{\R^d}{\infty}{c} \subseteq \Ltwomurho$. The measure $\murho$ has the density $\varrho>0$ which is continuously differentiable on $\R^d$ and $\mu=\lambda^d+\mS$, where $\lambda^d$ denotes Lebesgue measure and $\mS$ is a $\sigma$-finite measure on a closed Lebesgue nullset $A \in \Borel{\R^d}$. }
\end{equation}

In the following, for $O \subseteq \R^d$ open, $m \in \N$ and $p \in [1, \infty]$, we denote by $\text{H}^{m,p}_{\text{(loc)}}(O)$ the (local) Sobolev space of $m$ times weakly differentiable and (locally)  $p$ integrable functions \withrespectto $\lambda^d$ with (locally) $p$ integrable derivatives  \withrespectto $\lambda^d$ up to order $m$ on $O$.

\begin{prop}\label{prop: Domain is subset of Sob Space}\hfill
	For $f \in D(\Erho)$, its equivalence class $f \in \Ltwoloc$ is an element of the local Sobolev space $H^{1,2}_{\text{loc}}(\R^d)$. The weak derivatives are in $L^2(\R^d, \measdens{\lambda^d}{\varrho})$. 
\end{prop}
\begin{proof}
	We note that 
	\begin{equation*}
		\Lcaltwomurho \subseteq \Lcaltwo{\R^d}{\measdens{\lambda^d}{\varrho}} \subseteq \Lcaltwoloc,
	\end{equation*}
	where for a measure $\nu$ and $p \in [1, \infty]$, $\mathscr{L}^p\left(\R^d, \nu\right)$ denotes the set of pointwise defined $p$ integrable functions \withrespectto $\nu$. We distinguish this space from $L^p(\R^d, \nu)$.
	In particular, we have the inequalities
	\begin{align}\label{eq: Norm L2lambdarho lesseq L2murho}
		\norm{f}_{L^2\left(\R^d, \measdens{\lambda^d}{\varrho}\right)}\leq \norm{f}_{\Ltwomurho}
	\end{align}
	for $f \in \Lcaltwomurho$ and
	\begin{align}\label{eq: Norm L2loc lesseq L2lambdarho}
		\norm{\mathbbm{1}_{K}f}_{L^2(\R^d, \lambda^d)}\leq C \norm{f}_{L^2(\R^d, \measdens{\lambda^d}{\varrho})}
	\end{align}
	for $f \in \Lcaltwo{\R^d}{\measdens{\lambda^d}{\varrho}}$ and $K \subseteq \R^d$ compact with a constant $C \in \R$ depending on $\varrho$ and $K$ since $\varrho$ is continuous and strictly positive.
	Furthermore, the mappings $\Ltwomurho \to \Ltwoloc, f \mapsto f$ and $L^2(\R^d, \measdens{\lambda^d}{\varrho}) \to \Ltwoloc, f \mapsto f$ are well defined.\\
	Now let $f \in D(\Erho)$ and $(f_n)_{n \in \N}$ a sequence in $D$ that converges to $f$ \withrespectto $(\Erho_1)^{\frac{1}{2}}$. We fix $i \in \{1, \ldots, d\}$. With the definition of $\Erho$, we obtain that $(\partial_{x_i} f_n)_{n \in \N}$ forms a Cauchy sequence in $L^2(\R^d, \measdens{\lambda^d}{\varrho})$.
	Consequently, the sequence has a limit $f^{(i)} \in L^2(\R^d, \measdens{\lambda^d}{\varrho})$.  
	Let $\varphi \in \ContDiff{\R^d}{\infty}{c}$ and $K \subseteq \R^d$ compact such that supp$(\varphi) \subseteq K$. With the Hölder inequality and the estimates \eqref{eq: Norm L2loc lesseq L2lambdarho} and \eqref{eq: Norm L2lambdarho lesseq L2murho} we compute
	\begin{align}\label{eq: conv part one proof loc}
		&\abs{~\myint{\R^d}{f \partial_{x_i}\varphi\,}{\lambda^d}{}- \myint{\R^d}{f_n \partial_{x_i}\varphi\,}{\lambda^d}{}{}} 
		\leq \norm{\mathbbm{1}_K(f-f_n)}_{L^2(\R^d, \lambda^d)} \norm{\partial_{x_i}\varphi}_{L^2(\R^d, \lambda^d)} \nonumber\\
		&
		\leq  C \norm{f-f_n}_{\Ltwomurho} \norm{\partial_{x_i}\varphi}_{L^2(\R^d, \lambda^d)} \to 0
	\end{align}
	for $n \to \infty$ where $C \in \R$ is a constant depending on $K$ and $\varrho$.
	
	Moreover, with Equation \eqref{eq: Norm L2loc lesseq L2lambdarho} and the Hölder inequality we have
	\begin{align}\label{eq: conv part two proof loc}
		&\abs{~\myint{\R^d}{\partial_{x_i} f_n \varphi\,}{\lambda^d}{}- \myint{\R^d}{f^{(i)} \varphi\,}{\lambda^d}{}{}} 
		\leq \norm{\mathbbm{1}_K\left(\partial_{x_i} f_n-f^{(i)}\right)}_{L^2(\R^d, \lambda^d)} \norm{\varphi}_{L^2(\R^d, \lambda^d)}\nonumber\\&
		\leq C \norm{\partial_{x_i} f_n - f^{(i)}}_{L^2(\R^d, \measdens{\lambda^d}{\varrho})} \norm{\varphi}_{L^2(\R^d, \lambda^d)} \to 0
	\end{align}
	for $n \to \infty$ since $f^{(i)}$ is the limit of $(\partial_{x_i} f_n)_{n \in \N}$ in $L^2(\R^d, \measdens{\lambda^d}{\varrho})$.
	
	As  all $f_n$ are elements of $\Sobolev{{1,2}}{\R^d}$, Equation \eqref{eq: conv part one proof loc} and \eqref{eq: conv part two proof loc} yield
	\begin{align*}
		\myint{\R^d}{f \partial_{x_i}\varphi\,}{\lambda^d}{}{} 
		= \lim_{n \to \infty} \myint{\R^d}{f_n \partial_{x_i}\varphi\,}{\lambda^d}{}{}
		= \lim_{n \to \infty} (-1) \myint{\R^d}{\partial_{x_i} f_n \varphi\,}{\lambda^d}{}{}
		=(-1) \myint{\R^d}{f^{(i)} \varphi\,}{\lambda^d}{}{}
	\end{align*}
	and we conclude that $f \in \Ltwoloc$ is an element of $H^{1,2}_{\text{loc}}(\R^d)$ with weak derivative $f^{(i)} \in \Ltwoloc$.
\end{proof}

\begin{rem}
	Proposition \ref{prop: Domain is subset of Sob Space} yields the representation $\Erho(f,g) = \sum_{i=1}^d\myint{\R^d}{f^{(i)}g^{(i)}\,\varrho}{\lambda^d}{}$ for $f,g \in D(\Erho)$, where $f^{(i)}, g^{(i)}$ denote the weak derivatives of $f, g \in \localSobolev{1,2}{\R^d}$, respectively, for $i=1,\ldots d$.
\end{rem}

\begin{cond}\label{cond: d is one}
	Let $d=1$ and $(x_n)_{n \in \N}$ be a sequence in $\R$ with the property that for all $x \in \R$ there exists $\eps_x>0$ such that $\abs{B_{\eps_x}(x) \cap \{x_n \vcentcolon n \in \N\} }< \infty,$
	where $B_{\eps_x}(x) \vcentcolon = \{y \in \R \vcentcolon \abs{y-x}<\eps_x\}$.
	Define $\SetA \vcentcolon = \{x_n \vcentcolon n \in \N\}$ and $\mS(F)= \sum_{n \in \N}\delta_{x_n}(F)$ as sum of the Dirac measures in $x_n$ for all $F \in \Borel{\R}$.
\end{cond}

Note that under Condition \ref{cond: d is one}, $\SetA$ is closed, has $\lambda$-measure zero and $\mS$ is $\sigma$-finite.

\begin{prop}\label{prop: domain if $d=1$}
	Assume that Condition \ref{cond: d is one} holds true and $\varrho = \mathbbm{1}$ as constant one function.
	Then
	\begin{enumerate}[(i)]
		\item for $f \in D(\ErhoOne)$, the equivalence class of $f \in L^2(\R, \lambda)$ is an element of $H^{1,2}(\R)$ and
		\item if $\abs{\SetA}<\infty$ the unique continuous representative of $f$ in $H^{1,2}(\R)$ is an element of $\mathscr{L}^2(\R, \mu)$ with equivalence class in $D(\ErhoOne)$.
	\end{enumerate}
\end{prop}
\begin{proof}\hfill
	\begin{enumerate}[(i)]
		\item This part follows from Proposition \ref{prop: Domain is subset of Sob Space} and its proof with $\varrho$ as constant one function.
		\item Let $f \in H^{1,2}(\R)$. By a Sobolev embedding theorem ,see e.g.~\cite{AdamsSobSpaces}, Lemma 5.15 (p.107), there exists a continuous and bounded representative of $f$ denoted by $\widetilde{f}$ such that
		\begin{align*}
			\norm{\widetilde{f}}_{\text{sup}} \leq C \norm{f}_{H^{1,2}(\R)}
		\end{align*}
		for some $0<C<\infty$ independent of $f$. In particular, $\norm{\widetilde{f}}_{\LtwomuOneDim}< \infty$.
		By density, there is a sequence $(f_n)_{n \in \N}$ in $\ContDiff{\R}{\infty}{c} \subseteq H^{1,2}(\R)$ with $\norm{f_n-\widetilde{f}}_{H^{1,2}(\R)} \to 0$ as $n \to \infty$.
		This directly implies $\ErhoOne\left(f_n-f_m, f_n-f_m\right)\to 0$ as $n,m \to \infty$ and $\norm{f_n-\widetilde{f}}_{L^2(\R, \lambda)}\to0$ as $n \to \infty$.
		For any $x_k \in \SetA$, we have
		\begin{align*}
			\abs{f_n(x_k)-\widetilde{f}(x_k)} \leq \norm{f_n-\widetilde{f}}_{\text{sup}} \leq C \norm{f_n-f}_{H^{1,2}(\R)}
		\end{align*}
		because $f_n-\widetilde{f}$ is the continuous representative of $f_n-f \in H^{1,2}(\R)$  for all $n \in \N$.
		Therefore
		\begin{align*}
			\myint{\SetA}{(f_n(x)-\widetilde{f}(x))^2\,}{\mS}{}=\sum_{k=1}^{\abs{\SetA}}(f_n(x_k)-\widetilde{f}(x_k))^2 \to 0
		\end{align*}
		as $n \to \infty$. Thus, the equivalence class of $\widetilde{f}$ \withrespectto $\mu$ is an element of $D(\ErhoOne)$. 
	\end{enumerate}
\end{proof}

\begin{rem}
	\begin{enumerate}[(i)]
		\item 	We note that \ref{prop: domain if $d=1$}(b) does not hold in general if $\abs{\SetA}=\infty$. We choose $x_n \vcentcolon= \sqrt{\ln(\sqrt{n})}$ for $n \in \N$ and $f \in \Sobolev{{1,2}}{\R}$ with continuous representative $\widetilde{f}(x) \vcentcolon = \exp(-x^2)$ for $x \in \R$. Then
		\begin{align*}
			\myint{\SetA}{\widetilde{f}^2(x)\,}{\mS(x)}{}= \sum_{n \in \N}\widetilde{f}^2(x_n) = \sum_{n \in \N}\exp(-x_n^2)^2	= \sum_{n \in \N}\frac{1}{n}
		\end{align*} 
		what shows that $\widetilde{f}$ is not square integrable \withrespectto $\mu$.
		\item We note that two representatives of $f \in L^2(\R, \lambda)$ do not necessarily represent the same equivalence class \withrespectto $\measdens{\mu}{\varrho}$. Thus, the choice of the representative in $(b)$ is important. However,  the above mentioned Sobolev embedding theorem works only for $d = 1$ since $p=2$.
	\end{enumerate}
\end{rem}

\begin{prop}\label{prop: Diri from regular and strongly local}
	The symmetric closed form $(\Erho, D(\Erho))$ is a regular and strongly local Dirichlet form. 
\end{prop}
\begin{proof}
	To show that $(\Erho, D(\Erho))$ is a Dirichlet form, we check the conditions of \cite{MaRoeckner}, Proposition I.4.7 (p. 34) modified as in \cite{MaRoeckner}, Exercise I.4.8 (p.35) on $D$ and apply \cite{MaRoeckner}, Proposition I.4.10 (p. 35).\\
	The regularity follows from the fact that $D \subseteq D(\Erho)$ is dense and $\ContDiff{\R^d}{\infty}{c} \subseteq \ContDiff{\R^d}{0}{c}$ is dense \withrespectto $\norm{\cdot}_{\text{sup}}$. 
	
	To show that $(\Erho, D(\Erho))$ is strongly local, let $f$ and $g$ be in $D(\Erho)$ with essential support  \esssupp{f} and \esssupp{g} compact and such that $g$ is almost surely constant on a neighbourhood $V$ of \esssupp{f}. 
	From Proposition \ref{prop: Domain is subset of Sob Space}, we know that the equivalence class of $f$ \withrespectto $\lambda^d$ is in $ H^{1,2}_{\text{loc}}(\R^d)$. Additionally, $f^{(i)}=0\ \lambda^d$-- almost everywhere on $\R^d \setminus \esssupp{f} $.
	Furthermore, $g^{(i)}$ restricted to $V$ is zero in $L^2_{\text{loc}}(V, \lambda^d)$.
	Having this at hand, we conclude
	\begin{align*}
		\Erho(f,g) 
		= \sum_{i=1}^d\myint{\esssupp{f}}{f^{(i)}g^{(i)}\,\varrho}{\lambda^d}{}+ \sum_{i=1}^d \myint{\R^d \setminus \esssupp{f}}{f^{(i)}g^{(i)}\,\varrho}{\lambda^d}{} 
		=0.
	\end{align*}
\end{proof}

\begin{cond}\label{cond: rho ess bounded}
	Let $\varrho \in \ContDiff{\R^d}{1}{}$ be bounded and $\varrho>0$.
\end{cond}

\begin{prop}\label{prop: E is recurrent}
	Assume that $\varrho$ is continuously differentiable, $\varrho>0$ and integrable with respect to $\lambda^d$. Then, $(\Erho, D(\Erho))$ is recurrent and hence conservative.\\
	If $d=1$, then the integrability assumption on $\varrho$ can be replaced by boundedness, i.e.~one can assume Condition \ref{cond: rho ess bounded}.
\end{prop}

\begin{proof}
	The aim is to apply \cite{Fukushima}, Theorem 1.6.3 (ii) (p. 58). 
	We define
	\begin{align*}
		\Psi \vcentcolon \R^d \to \R, 
		x \mapsto \begin{cases} N \exp\left(-\frac{1}{1-(2 \eucnorm{x})^2}\right), &\text{ if } \eucnorm{x}< \nicefrac{1}{2}\\
			0, &\text{ else}\end{cases},
	\end{align*}
	where $N$ is a normalizing factor such that $\myint{\R^d}{\Psi\,}{\lambda^d}{}=1$. We set $M \vcentcolon= B_2(0)$ and define 
	\begin{align*}
		\varphi\vcentcolon \R^d \to \R, 
		x \mapsto \myint{\R^d}{\Psi(x-y)\mathbbm{1}_{M}(y)\,}{\lambda^d(y)}{}.
	\end{align*}
	Furthermore, we define $\varphi_n(x) \vcentcolon= \varphi\left(\frac{x}{n}\right)$ for all $x \in \R^d$. We note that $(\varphi_n)_{n \in \N} \subseteq \ContDiff{\R^d}{\infty}{c}$ with supp($\varphi_n$)$\subseteq \overline{B_{3n}(0)}$ for all $n \in \N$. 
	For $x \in \R^d$, we have $\varphi_n(x) \to 1$ as $n \to \infty$
	and 
	\begin{align*}
		\partial_{x_i}  \varphi_n(x) = \partial_{x_i} \varphi\left(\frac{x}{n}\right)
		= \frac{1}{n} \myint{\R^d}{\partial_{x_i} \Psi\left(\frac{x}{n}-y\right)\mathbbm{1}_M(y)\,}{\lambda^d(y)}{}
	\end{align*}
	for $i=1, \ldots d$.
	Since there exists a constant $C \in \R_{>0}$ that bounds $\partial_{x_i} \Psi$ for all $i=1, \ldots d$, we have for all those $i$ that 
	$\abs{\partial_{x_i}  \varphi_n(x)} \leq \frac{1}{n}C \lambda^d(M)$ for all $x \in \R^d$.
	We obtain
	\begin{align}\label{eq: show conv in recurrence}
		\Erho(\varphi_n, \varphi_n) 
		\leq d \myint{\overline{B_{3n}(0)}}{\left(\frac{1}{n}C\lambda^d(M)\right)^2\,\varrho(x)}{\lambda^d}{}
		\leq d \left(\frac{1}{n}C\lambda^d(M)\right)^2 \myint{\overline{B_{3n}(0)}}{\varrho(x)\,}{\lambda^d}{}.
	\end{align}
	If $\varrho$ is integrable with respect to $\lambda^d$, then \eqref{eq: show conv in recurrence} is zero in the limit for $n \to \infty$.
	If $d=1$, then $\lambda(\overline{B_{3n}(0)})=6n$ and the boundedness of $\varrho$ implies the same result.
\end{proof}

%% file: sections/SectionLipschitzDomainGenerator_DBM_ARXIV.tex
\section{The associated Generator}
In this section, we give an explicit representation of the generator of the Dirichlet form $(\Erho, D(\Erho))$ given in \eqref{defn: Erho} on the space $\ContDiff{\R^d}{0}{c} \cap \ContDiff{\R^d\setminus\SetA}{2}{b} \subseteq \Ltwomurho$, where $\ContDiff{\R^d\setminus\SetA}{2}{b}$ denotes the space of twice continuously differentiable functions on $\R^d\setminus \SetA$ that are bounded and have bounded derivatives up to order $2$, for special choices of $\SetA$ and $\mS$. In dimension $d=1$, we assume that Condition \ref{cond: d is one} holds true. For $d \geq 2$, we assume that $A$ is given as Lipschitz boundary (see Condition \ref{cond: Lipschitz boundary}). We therefore start with the following definition.

\begin{definition}[{\cite{ALTLAEnglish}, A8.2 (p.259)}]\label{def: Lipschitz boundary}
	Let $d \geq 2$ and $\Lipschitz \subseteq \R^d$ be open and bounded. We say that $\Lipschitz$ has Lipschitz boundary if $\partial\Lipschitz$ can be covered by finitely many open sets $U_1, \ldots U_l$ such that $\partial\Lipschitz\cap U_k$ for $k=1, \ldots l$ is the graph of a Lipschitz continuous function with $\Lipschitz\cap U_k$  in each case lying on one side of this graph. This means the following: There exists $l \in \N$ and for $k = 1,...,l$ an orthonormal basis $B_k = \{e_1^k, \ldots e_d^k\}$ \withrespectto the Euclidean scalar product, a real number $h_k \in \R_{>0}$, an open cuboid $ I_k \vcentcolon=\bigtimes_{j=1}^{d-1}(a_j^k, b_j^k)$ and $\varphi_k \vcentcolon \R^{d-1} \to \R$ Lipschitz continuous such that
	with the notation $\vdimlower\vcentcolon=(v_1, \ldots, v_{d-1}) \text{ for }v=\sum_{j=1}^{d}v_j e_j^k$
	it holds 
	\begin{align*}
		U_k = \left\{\sum_{j=1}^{d} v_j e_j^k \in \R^d \vcentcolon v_{\textbf{.} d-1} \in I_k \text{ and }\abs{v_d-\varphi_k(\vdimlower)}<h_k\right\}
	\end{align*}
	and for $v \in U_k$ we have $v \in \partial\Lipschitz$ if $v_d =\varphi_k(\vdimlower)$, $v \in \Lipschitz$ if $0< v_d-\varphi_k(\vdimlower)< h_k$ and $v \in \LipschitzComplement$ if $-h_k < v_d -\varphi_k(\vdimlower)<0 $.
	Furthermore, $\partial\Lipschitz \subseteq \bigcup_{k=1}^lU_k$.
\end{definition}

We denote the canonical basis of $\R^d$ by $\{e_1, \ldots e_d\}$. 

\begin{lem}\label{lem: Partition of unity Rd}
	Let $\Lipschitz$ be open and bounded with Lipschitz boundary $\partial \Lipschitz$. Then there exists a locally finite open cover of $\R^d$ denoted by $(U_k)_{k \in \N}$ that extends the open cover of $\partial\Lipschitz$ in a way that the sets added to the open cover of $\partial \Lipschitz$ are disjoint with $\partial\Lipschitz$.
	Moreover, there exists  a corresponding partition of unity $(\eta_k)_{k \in \N}$.
\end{lem}

\begin{proof}
	Let $d \geq 2$ and let $U_1, \ldots, U_l$ be a finite open cover of $\partial \Lipschitz$ as in the Definition of Lipschitz boundary \ref{def: Lipschitz boundary}. We choose $K \in \N$ such that $B_{K-1}(0) \supseteq \Lipschitz$. 
	For $n \geq 2$, we define $U_{l+n}= \vcentcolon B_{K+n}(0) \setminus B_{K+n-2}(0)$.
	Since there exists some $\xi>0$ such that
	\begin{equation*}
		G \vcentcolon= \bigcup\limits_{i=1}^l \overline{\left\{\sum_{j=1}^{d-1} z_i e_j^k + \left(\varphi_k(z_{\textbf{.}d-1})+ \delta \right)e_d^k \vcentcolon (z_{\textbf{.}d-1}) \in\bigtimes_{j=1}^{d-1}(a_j^k+\xi, b_j^k-\xi)\text{ and }\delta \in [- \xi, \xi]\right\}}
	\end{equation*}
	is a subset of $\bigcup\limits_{i=1}^l U_k$, the set $U_{l+1}\vcentcolon = G^{\text{C}}\cap B_{K+1}(0)$
	completes $(U_k)_{k \in \N}$ to a locally finite open cover of $\R^d$. Since all $U_k$ are bounded,  \cite{ALTLAEnglish}, Proposition (p.118) ensures the existence of a corresponding partition of unity $(\eta_k)_{k \in \N}$.
\end{proof}

\begin{cond}\label{cond: Lipschitz boundary}
	Let $d \geq 2$ and let $\Lipschitz$ be an open and bounded set with Lipschitz boundary. We set $\SetA \vcentcolon = \partial \Lipschitz$. Moreover, we define
	\begin{align*}
		\mS \vcentcolon \Borel{\SetA} \to [0, \infty],\ F \mapsto \myint{\SetA}{\mathbbm{1}_F\,}{H^{d-1}}{} \vcentcolon= \sum_{k=1}^l \myint{\SetA}{\eta_k \mathbbm{1}_F\,}{H^{d-1}}{}
	\end{align*}
	in the sense of a boundary integral (see \cite{ALTLAEnglish}, A8.5 (p.263)), where 
	\begin{equation*}
		\myint{\SetA}{\eta_k \mathbbm{1}_F\,}{H^{d-1}}{} = \myint{I_k}{\eta_k \mathbbm{1}_F(\vdimlower, \varphi_k(\vdimlower))\sqrt{1+\sum_{i=1}^{d-1}(\varphi_k^{(i)})^2(\vdimlower)}\,}{\lambda^{d-1}(\vdimlower)}{}.
	\end{equation*}
	The measure $H^{d-1}$ is called $(d-1)$-dimensional Hausdorff measure.
\end{cond}
Since $\varphi_k$ is Lipschitz continuous on $\R^{d-1}$, it holds $\varphi_k \in \text{H}^{1, \infty}_{\text{loc}}(\R^{d-1})$. For $i \in \{1, \ldots, d\}$, we denote by $\varphi_k^{(i)}$ the weak derivative of $\varphi_k$ in direction $e_i$ in $L^{\infty}_{\text{loc}}(\R^{d-1})$. 
\noindent We note that under Condition \ref{cond: Lipschitz boundary}, $\SetA$ is a closed $\lambda^d$ null set and $\mS$ is finite on $\Borel{\SetA}$.

\begin{lem}\label{lem: hplus and hminus}
	Let $h \in \ContDiff{\R^d\setminus\SetA}{1}{b}$, \ie a continuous function on $\R^d\setminus\SetA$ that is bounded, differentiable up to order $1$ and has bounded derivatives.
	\begin{enumerate}[(i)]
		\item Assume that Condition \ref{cond: d is one} is satisfied. Then both functions $\extension{h}{r}, \extension{h}{l} \vcentcolon \R \to \R$ given by
		\begin{align*}
			\extension{h}{r}(x) \vcentcolon= \begin{cases}
				h(x), \text{ if } x \in \R\setminus \SetA\\
				\lim\limits_{n \to \infty} h(y_n) \text{ for an arbitrary sequence }(y_n)_{n \in \N} \subseteq \R\setminus\SetA\text{ with }x<y_n\\ \text{ for all }n \in \N,\text{if } x \in \SetA
			\end{cases} 
		\end{align*}
		and
		\begin{align*}
			\extension{h}{l}(x) \vcentcolon=
			\begin{cases}
				h(x), \text{ if } x \in \R\setminus\SetA\\
				\lim\limits_{n \to \infty} h(y_n)\text{ for an arbitrary sequence }(y_n)_{n \in \N} \subseteq \R\setminus\SetA\text{ with }y_n<x \\\text{ for all }n \in \N, \text{if } x \in \SetA
			\end{cases}
		\end{align*}
		for $x \in \R$ are well defined, bounded and measurable.
		\item Assume that Condition \ref{cond: Lipschitz boundary} is satisfied. Then both functions $\extension{h}{\Lipschitz},\extension{h}{\LipschitzComplement} \vcentcolon\R^d \to \R$
		\begin{align*}
			\extension{h}{\Lipschitz}(x)\vcentcolon= \begin{cases}
				h(x), \text{ if } x \in \R^d\setminus\SetA\\
				\lim\limits_{n \to \infty} h(x^{(n)}) \text{ for an arbitrary sequence }(x^{(n)})_{n \in \N} \subseteq \Lipschitz, \text{ if } x \in \SetA
			\end{cases} 
		\end{align*}
		and
		\begin{align*}
			\extension{h}{\LipschitzComplement}(x)\vcentcolon=
			\begin{cases}
				h(x), \text{ if } x \in \R^d\setminus\SetA\\
				\lim\limits_{n \to \infty} h(x^{(n)})\text{ for an arbitrary sequence }(x^{(n)})_{n \in \N} \subseteq (\overline{\Lipschitz})^{\text{C}}, \text{ if }x \in\SetA
			\end{cases}
		\end{align*}
		for $x \in \R^d$ are well defined, bounded and measurable.
	\end{enumerate}
\end{lem}

\begin{proof}
	Let $h \in \ContDiff{\R^d\setminus\SetA}{1}{b}$.
	\begin{enumerate}[(i)]
		\item Assume that Condition \ref{cond: d is one} is satisfied. Let $x \in \SetA$.  Then there exists $\eps_x>0$ such that $(x, x+\eps_x) \subseteq \R\setminus\SetA$. We choose an arbitrary sequence $(y_n)_{n \in \N} \subseteq (x, x+\eps_x)$ with $\lim_{n \to \infty}y_n=x$. For any combination of $n,m \in \N$ with $y_n<y_m$ we have $h\vcentcolon [y_n, y_m] \to \R$ is continuous and differentiable on $(y_n, y_m)$. The mean value theorem implies $\abs{h(y_n)-h(y_m)} \leq \norm{h'}_{\sup} \abs{y_n-y_m}$. Thus, $(h(y_n))_{n \in \N}$ is a Cauchy sequence in $\R$ and there exists $\lim_{n \to \infty}h(y_n)$. The same argument also shows independence of the limit of the sequence $(y_n)_{n \in \N}$ and consequently, there exists $\extension{h}{r}(x) \vcentcolon = \lim_{n \to \infty}h(y_n)$ for an arbitrary sequence $(y_n)_{n \in \N}$. A similar proof shows the existence of  $\extension{h}{l}$. Note that by definition, $\extension{h}{r}$ and $\extension{h}{l}$ are bounded by the same constants as is $h$. Since $\extension{h}{r}$ is a right-continuous and $\extension{h}{l}$ is a left-continuous function on $\R$, they are $\Borel{\R}/\Borel{\R}$-measurable. 
		\item Assume that Condition \ref{cond: Lipschitz boundary} is satisfied.
		Let $C \in \R_{\geq 1}$ such that $\eucnorm{J h (x)} \leq C$ for all $x \in \R^d \setminus \SetA$, denoting by $J$ the corresponding Jacobian matrix. We show that $\extension{h}{\Lipschitz}$ is well defined. 
		Let $x \in \SetA$ and $(x^{(n)})_{n \in \N} \subseteq \Lipschitz$ with $\lim_{n \to \infty}x^{(n)} = x$. 
		There is $k \in \N$ with $x \in \SetA \cap U_k$ and \withoutloss $(x^{(n)})_{n \in \N} \subseteq U_k \cap \Lipschitz$. Let $0<\eps<h_k$ be arbitrary.
		The Lipschitz continuity of $\varphi_k$ on $\R^{d-1}$ implies the existence of $0<\delta<\frac{\eps}{2C}$ such that $\abs{\varphi_k(v)-\varphi_k(w)}< \frac{\epsilon}{4C}$ for all $v,w \in \R^{d-1}$ with $\eucnorm{v-w}<\delta$. Choose $N \in \N$ such that $\eucnorm{x^{(n)}-x}<\frac{\delta}{4}$ holds for all $n \geq N$.\\
		Now let $n,m >N$ be arbitrary.
		Define
		\begin{align*}
			\gamma_1 &\vcentcolon \left[0, \varphi_k(x_{\textbf{.}d-1})+\frac{\eps}{4C}-x_d^{(n)}\right] \to \R^d, s \mapsto x^{(n)} + s e_d^k\\
			\gamma_2 &\vcentcolon [0,1]\to \R^d, s \mapsto \sum_{i=1}^{d-1}(x_i^{(n)} + s(x^{(m)}_i-x_i^{(n)}))e_i^k + \left(\varphi_k(x_{\textbf{.}d-1})+\frac{\eps}{4C}\right)e_d^k\\
			\gamma_3 &\vcentcolon \left[0, \varphi_k(x_{\textbf{.}d-1})+\frac{\eps}{4C}-x_d^{(m)}\right] \to \R^d, s \mapsto \sum_{i=1}^{d-1}x_i^{(m)}e_i^k+ \left(\varphi_k(x_{\textbf{.}d-1})+\frac{\eps}{4C}-s\right)e_d^k.
		\end{align*}
		Note that the choice of $n$ and $m$ implies $\varphi_k(x_{\textbf{.}d-1})+\frac{\eps}{4C}-x_d^{(n)}>0$ and $\varphi_k(x_{\textbf{.}d-1})+\frac{\eps}{4C}-x_d^{(m)}>0$ and that the images of all functions $\gamma_i, i=1,\ldots, 3$ are a subset of $U_k \cap \Lipschitz \setminus  \SetA$.
		We estimate with the mean value theorem 
		\begin{align*}
			&\abs{h(x^{(n)})-h(x^{(m)})} \leq \abs{h(\gamma_1(0)) - h\left(\gamma_1\left(\varphi_k(x_{\textbf{.}d-1})+\frac{\eps}{4C}-x_d^{(n)}\right)\right)}\\
			&+ \abs{h(\gamma_2(0))-h(\gamma_2(1))}
			+ \abs{h(\gamma_3(0)) - h\left(\gamma_3\left( \varphi_k(x_{\textbf{.}d-1})+\frac{\eps}{4C}-x_d^{(m)}\right)\right)}\\
			&\leq  \eucnorm{Jh(\gamma_1(\xi_1))J\gamma_1(\xi_1)}\abs{\varphi_k(x_{\textbf{.}d-1})+\frac{\eps}{4C}-x_d^{(n)}}
			\\&+  \eucnorm{Jh(\gamma_2(\xi_2))J\gamma_2(\xi_2)}\abs{1-0}+  \eucnorm{Jh(\gamma_3(\xi_3))J\gamma_3(\xi_3)}\abs{ \varphi_k(x_{\textbf{.}d-1})+\frac{\eps}{4C}-x_d^{(m)}}
			\\&\leq  C \left(\frac{\eps}{4C}+\frac{\delta}{4}\right)
			+ C \frac{\delta}{2}
			+ C \left(\frac{\eps}{4C}+\frac{\delta}{4}\right)= C \left(\delta + \frac{\eps}{2C}\right) < \eps.
		\end{align*}
		Therefore, $h(x^{(n)}))_{n \in \N}$ is a Cauchy sequence and there exists $\lim_{n \to \infty}h\left(x^{(n)}\right) \in \R$. With the same arguments as above, the limit is independent of the chosen sequence in $\Lipschitz$. Thus, there exists $\extension{h}{\Lipschitz}(x) \vcentcolon = \lim_{n \to \infty} h(x^{(n)})$ for an arbitrary sequence $(x^{(n)})_{n \in \N} \subseteq \Lipschitz$.
		The restriction $\restrictfunc{\extension{h}{\Lipschitz}}{\overline{U}}$ of $\extension{h}{\Lipschitz}$ to $\overline{U}$ is continuous and consequently, $\Borel{\overline{U}}$ measurable. Thus, $\extension{h}{\Lipschitz} = \mathbbm{1}_{\overline{U}}\left(\restrictfunc{\extension{h}{\Lipschitz}}{\overline{U}}\right) + \mathbbm{1}_{\LipschitzComplement}h$ is $\Borel{\R^d}$ measurable.
		A similar proof shows the existence and measurability of  $\extension{h}{\LipschitzComplement}$. 
		Note that by definition, $\extension{h}{\Lipschitz}$ and $\extension{h}{\LipschitzComplement}$ are bounded by the same constants as is $h$.
	\end{enumerate}
\end{proof}	

\begin{prop}\label{prop: calculation of Uz in $d=1$}
	Assume that Condition \ref{cond: d is one} is satisfied. Let $h \in \ContDiff{\R\setminus\SetA}{1}{b}$ and $g \in \ContDiff{\R}{\infty}{c}$. Let $(U_z)_{z \in \Z}$ be the locally finite open cover of $\R$ defined by $U_z \vcentcolon =(z, z+2)$ and $(\eta_z)_{z \in \Z}$ be the corresponding partition of unity.
	\begin{enumerate}[(i)]
		\item Let $z \in \Z$ with $U_z \cap \SetA = \emptyset$. Then 
		\begin{align*}
			\myint{U_z}{\eta_z g' h \,\varrho}{\lambda}{} = -\myint{U_z}{(\eta_k h \varrho)'g\,}{\lambda}{}.
		\end{align*}
		\item Let $z \in \Z$ with $U_z \cap \SetA \neq \emptyset$.  Then 
		\begin{align*}
			\myint{U_z\setminus \SetA}{\eta_z g' h \,\varrho}{\lambda}{}  = - \myint{U_z\setminus \SetA}{(\eta_zh \varrho)'g\,}{\lambda}{} + \myint{\SetA}{\eta_z g\left(\extension{h}{l}- \extension{h}{r}\right)\,\varrho}{\mS}{}.
		\end{align*}
	\end{enumerate}
\end{prop}

\begin{proof}
	Let $h \in \ContDiff{\R\setminus\SetA}{1}{b}$ and $g \in \ContDiff{\R}{\infty}{c}$.
	\begin{enumerate}[(i)]
		\item Choose $z \in \Z$ with $U_z \cap \SetA = \emptyset$. Lebesgue dominated convergence and integration by parts for Riemann integrals yield
		\begin{align*}
			&\myint{U_z}{\eta_z g' h \, \varrho}{\lambda}{} = \lim_{\epsilon\to 0}\myint{[z+\epsilon, z+2-\epsilon]}{\eta_zg'h\,\varrho}{\lambda}{}
			= -\myint{U_z}{(\eta_zh\varrho)'g\,}{\lambda}{} \\+ &\lim_{\epsilon \to 0} \left(\eta_z g h \varrho(z+2-\varepsilon) - \eta_z g h \varrho (z+\epsilon)\right)
			=  -\myint{U_z}{(\eta_zh\varrho)'g\,}{\lambda}{}, 
		\end{align*}
		taking into account that $\eta_z$ has compact support in $U_z$.
		\item Let $z \in \Z$ with $U_z \cap \SetA \neq \emptyset$. The assumption on $\SetA$ and the definition of $U_z$ imply that  $U_z \cap \SetA$ is finite. We denote its elements in increasing order by $x_1^z, \ldots, x_{m_z}^z$ and set $x_0^z \vcentcolon = z$ and $x_{m_z+1}^z \vcentcolon = z+2$. Then
		$U_z \setminus \SetA =  \bigcup_{i=0}^{m_z}(x_i^z, x_{i+1}^z)$.
		We have with Lebesgue dominated convergence and integration by parts for Riemann integrals
		\begin{align*}
			&\myint{U_z\setminus \SetA}{\eta_z g' h \, \varrho}{\lambda}{} = \sum_{i=0}^{m_z}\lim_{\epsilon\to 0}\myint{[x_i^z+\epsilon, x_{i+1}^z-\epsilon]}{\eta_zg'h\,\varrho}{\lambda}{}\\
			&= \sum_{i=0}^{m_z}\lim_{\epsilon\to 0}\myint{[x_i^z+\epsilon, x_{i+1}^z-\epsilon]}{-(\eta_zh \varrho)'g\,}{\lambda}{} + \eta_zh\varrho g(x_{i+1}^z-\epsilon) - \eta_zh\varrho g(x_{i}^z+\epsilon)\\
			&= -\myint{U_z\setminus \SetA}{(\eta_zh \varrho)'g\,}{\lambda}{} + \sum_{i=0}^{m_z} \eta_z g \varrho(x_i^z)\left(\extension{h}{l}(x_i^z) - \extension{h}{r}(x_i^z)\right)\\&
			= -\myint{U_z\setminus \SetA}{(\eta_zh \varrho)'g\,}{\lambda}{} + \myint{\SetA}{\eta_z g\left(\extension{h}{l}- \extension{h}{r}\right)\,\varrho}{\mS}{}.
		\end{align*} 
	\end{enumerate}
	
\end{proof}

We define the outer normal of a Lipschitz boundary to obtain a generalization of Proposition \ref{prop: calculation of Uz in $d=1$} under Condition \ref{cond: Lipschitz boundary}.

\begin{definition}[{\cite{ALTLAEnglish}, A8.5 item (3) (p.263)}]
	Let $\Lipschitz$ be open and bounded with Lipschitz boundary.
	In the notation of Definition \ref{def: Lipschitz boundary}, the outer normal to $\Lipschitz$ at the point $x \in \partial \Lipschitz$ is defined by
	\begin{align*}
		\nu_{\Lipschitz}(x) \vcentcolon = \left(1+\abs{\nabla \varphi_k(x)}^2\right)^{-\frac{1}{2}}\left(\sum_{i=1}^{l-1}\varphi_k^{(i)}e_i^k-e_d^k\right)
	\end{align*}
	where $\nabla \varphi_k(x)$ denotes the gradient in the weak sense. 
\end{definition}

\begin{prop}\label{prop: Calculation of Uk on d greater two}
	Assume that Condition \ref{cond: Lipschitz boundary} is satisfied. 
	Let $h \in \ContDiff{\R^d\setminus \SetA}{1}{b}$ and $g \in \ContDiff{\R^d}{\infty}{c}$. Further, let $(U_k)_{k \in \N}$ be the locally finite open cover of $\R^d$  and $(\eta_k)_{k \in \N}$ the corresponding partition of unity constructed in Lemma \ref{lem: Partition of unity Rd}.
	\begin{enumerate}[(i)]
		\item If $k \in \N$ with $U_k \cap \SetA= \emptyset$, then	for $i\in \{1, \ldots d\}$
		\begin{align*}
			\myint{U_k}{ \partial_{x_i}g \eta_k h \, \varrho(x)}{\lambda^d(x)}{}= - \myint{U_k}{g\ \partial_{x_i}\left(\eta_k h \varrho\right)(x)\,}{\lambda^d(x)}{}.
		\end{align*}
		\item If $k \in \N$ with $U_k \cap \SetA \neq \emptyset$, then for $\extension{h}{\Lipschitz}, \extension{h}{\LipschitzComplement}$ defined as in Lemma \ref{lem: hplus and hminus} it holds 
		\begin{align*}
			\myint{U_k\setminus\SetA}{\partial_{x_i}g \eta_k h \,\varrho(x)}{\lambda^d(x)}{} = - &\myint{U_k\setminus \SetA}{g\ \partial_{x_i}\left(\eta_k h \varrho\right)(x)\,}{\lambda^d(x)}{}\\+ &\myint{\SetA}{g \eta_k  \left(\extension{h}{\Lipschitz}-\extension{h}{\LipschitzComplement}\right)(x)\eucscalar{\nu_{\Lipschitz}}{e_i}(x)\, \varrho}{H^{d-1}(x)}{}.
		\end{align*}
	\end{enumerate}
\end{prop}

\begin{proof}
	Let $h \in \ContDiff{\R^d\setminus \SetA}{1}{b}$ and $g \in \ContDiff{\R^d}{\infty}{c}$.
	\begin{enumerate}[(i)]
		\item Let $k \in \N$ with $U_k \cap \SetA= \emptyset$ and $i \in \{1,\ldots, d\}$. Then, $h \varrho \eta_k \in \ContDiff{U_k}{1}{c}\subseteq H^{1,2}(U_k)$  and the weak Gauss Theorem (compare \cite{ALTLAEnglish}, A8.8 (p.270-271)) yields
		\begin{align*}
			\myint{U_k}{ \partial_{x_i}g\ \eta_k h\, \varrho(x)}{\lambda^d(x)}{}= - &\myint{U_k}{g\ \partial_{x_i}\left(\eta_k h \varrho\right)(x)\,}{\lambda^d(x)}{}\\
			+ &\myint{ \partial U_k}{g \eta_k h \eucscalar{\nu_{\Lipschitz}}{e_i}(x)\,\varrho}{H^{d-1}(x)}{}.
		\end{align*}
		\item  Let $k \in \N$ with $U_k \cap \SetA\neq \emptyset$. Then there exists an orthonormal basis $B_k = \{e_1^k, \ldots e_d^k\}$, $h_k \in \R_{>0}$, an open cuboid $ I_k \vcentcolon=\bigtimes_{j=1}^{d-1}(a_j^k, b_j^k)$ and $\varphi_k \vcentcolon \R^{d-1} \to \R$ Lipschitz continuous such that 
		\begin{align*}
			U_k = \left\{\sum_{j=1}^{d} v_j e_j^k \in \R^d \vcentcolon v_{\textbf{.} d-1} \in I_k \text{ and }\abs{v_d-\varphi_k(v_{\textbf{.} d-1})}<h_k\right\}
		\end{align*}
		and $\SetA\cap U_k =  \left\{\sum_{j=1}^{d-1} v_j e_j^k + (\varphi_k(v_{\textbf{.} d-1})) e_d^k\in \R^d  \vcentcolon v_{\textbf{.} d-1} \in I_k\right\}.$
		For $0<\eps<h_k$ define 
		$Z_{k, \varepsilon}^1\vcentcolon = \bigcup_{v \in I_k} \{v\} \times (\varphi_k(v)+ \eps, \varphi_k(v)+ h_k)$,
		$Z_{k, \varepsilon}^2  \vcentcolon =\bigcup_{v \in I_k} \{v\} \times (\varphi_k(v)-h_k, \varphi_k(v)-\eps)$ and 
		$R_{k, \varepsilon} \vcentcolon =\bigcup_{v \in I_k} \{v\} \times [\varphi_k(v)-\eps, \varphi_k(v)+\eps]$.
		\\Then $Z_{k, \eps}^1 \cupdot Z_{k, \eps}^2 \uparrow \bigcup_{v \in I_k} \{v\} \times ((\varphi_k(v)-h_k, \varphi_k(v)+h_k)\setminus \{\varphi_k(v)\}) = \vcentcolon Z_k$. 
		
		With the substitution formula and Lebesgue dominated convergence we calculate for $i \in \{1, \ldots, d\}$
		\begin{align}
			&\myint{U_k\setminus \SetA}{\partial_{x_i}g\ \eta_k h \varrho(x)\,}{\lambda^d(x)}{}
			= \myint{Z_k}{\partial_{x_i}g\ \eta_k h \varrho\left(T_k(v)\right)\,}{\lambda^d(v)}{}\nonumber
			\\= &\lim_{\eps \to 0} \myint{Z_{k, \eps}^1 \cupdot Z_{k, \eps}^2}{\partial_{x_i}g\ \eta_k h \varrho\left(T_k(v)\right)\,}{\lambda^d(v)}{}, \text{ where $T_k(v) = \sum_{j=1}^d v_j e_j^k$ for $v \in \R^d$.} \label{eq: thm calculation of Uk Zerlegung mit LDC}
		\end{align}
		For some fixed $0<\eps<h_k$  we calculate with the weak Gauss Theorem 
		\begin{align}
			&\myint{Z_{k, \eps}^1 \cupdot Z_{k, \eps}^2}{\partial_{x_i}g \eta_k h \varrho\left(T_k(v)\right)\,}{\lambda^d(v)}{}
			\nonumber\\
			&= -  \myint{Z_{k, \eps}^1 \cupdot Z_{k, \eps}^2}{g \circ T_k(v) \left(\sum_{j=1}^d (e_j^k)_i \partial_{v_j}(\eta_k h \varrho \circ T_k) \right)(v)\,}{\lambda^d(v)}{}\label{eq: thm calculation of Uk Calculation with boundary term 1}\\
			&+   \sum_{l=1}^2\sum_{j=1}^d (e_j^k)_i \myint{\partial Z_{k, \eps}^l}{(\eta_k h \varrho g) \circ T_k(v)\eucscalar{\nu_{Z_{k, \eps}^l}}{e_j^k}(v)\,}{H^{d-1}(v)}{}\label{eq: thm calculation of Uk Calculation with boundary term 2} .
		\end{align}
		Lebesgue dominated convergence and  the substitution formula allow to represent Equation \eqref{eq: thm calculation of Uk Calculation with boundary term 1} in the limit as 
		\begin{align}
			&-\lim_{\eps \to 0}\myint{Z_{k, \eps}^1 \cupdot Z_{k, \eps}^2}{g \circ T_k(v) \left(\sum_{j=1}^d (e_j^k)_i \partial_{v_j}(\eta_k h \varrho \circ T_k) \right)(v)\,}{\lambda^d(v)}{}\nonumber\\
			= &- \myint{U_k\setminus \SetA}{g(x) \partial_{x_i}(\eta_k h \varrho )(x)\,}{\lambda^d(x)}{}\label{eq: thm calculation of UK haupt term done}.
		\end{align}
		Considering the boundary integrals in Equation \eqref{eq: thm calculation of Uk Calculation with boundary term 2}, we note that the integrand is zero on $\partial Z_{k, \eps}^1 \cap \partial U_k$ and $\partial Z_{k, \eps}^2 \cap \partial U_k$.
		We obtain for the integral over $\partial Z_{k, \eps}^1$
		\begin{align*}
			&\sum_{j=1}^d (e_j^k)_i \myint{\partial Z_{k, \eps}^1}{(\eta_k h \varrho g) \circ T_k(v)\eucscalar{\nu_{Z_{k, \eps}^l}}{e_j^k}(v)\,}{H^{d-1}(v)}{}\\=
			&\myint{I_k}{\eta_k h \varrho g(T_k(s, \varphi_k(s)+ \eps))\left(\sum_{j=1}^{d-1}  (e_j^k)_i \partial_{v_j} \varphi_k (s) - (e_d^k)_i\right)\,}{\lambda^{d-1}(s)}{}
		\end{align*}
		In the limit we have with Lebesgue dominated convergence 
		\begin{align}
			&\lim_{\eps \to 0}\myint{I_k}{\eta_k h \varrho g(T_k(s, \varphi_k(s)+ \eps))\left(\sum_{j=1}^{d-1}  (e_j^k)_i \partial_{v_j} \varphi_k (s) - (e_d^k)_i\right)\,}{\lambda^{d-1}(s)}{} \nonumber\\
			&= \myint{I_k}{\eta_k \extension{h}{\Lipschitz} \varrho g (T_k(s, \varphi_k(s)) \left(\sum_{j=1}^{d-1}  (e_j^k)_i \partial_{v_j} \varphi_k (s) - (e_d^k)_i\right)\,}{\lambda^{d-1}(s)}{}\nonumber\\
			&=  \myint{\SetA}{\eta_k \extension{h}{\Lipschitz} g (x) \eucscalar{\nu_{\Lipschitz}}{e_i}(x)\,\varrho(x)}{H^{d-1}(x)}{}.\label{eq: thm calculation of UK boudary term 1 done}
		\end{align}
		where  $\extension{h}{\Lipschitz}$ is defined as in Lemma \ref{lem: hplus and hminus}. The computation for Equation \eqref{eq: thm calculation of Uk Calculation with boundary term 2} is similar, reflecting the vector $e_d^k$ with the transformation $S_k(v) = \sum_{j=1}^{d-1}v_j e_j - v_d e_d$
		for $v \in \R^d$ and taking inner derivatives into account. We obtain
		\begin{align*}
			&\sum_{j=1}^d (e_j^k)_i \myint{\partial Z_{k, \eps}^2}{(\eta_k h \varrho g) \circ T_k(v)\eucscalar{\nu_{Z_{k, \eps}^2}}{e_j^k}(v)\,}{H^{d-1}(v)}{}\\
			= &\myint{I_k}{\eta_k h \varrho g(T_k(S_k(s, -\varphi_k(s)+\eps)))  \left(-\sum_{j=1}^{d-1}  (e_j^k)_i \partial_{v_j} \varphi_k (s) + (e_d^k)_i\right)\,}{\lambda^{d-1}(s)}{}\\
			= &\myint{I_k}{\eta_k h \varrho g(T_k(s, \varphi_k(s)-\eps)) \left(-\sum_{j=1}^{d-1}  (e_j^k)_i \partial_{v_j} \varphi_k (s) + (e_d^k)_i\right)\,}{\lambda^{d-1}(s)}{}.
		\end{align*}
		Analogously it holds
		\begin{align}
			&\lim_{\eps \to 0}\myint{I_k}{\eta_k h \varrho g(T_k(s, \varphi_k(s)-\eps)) \left(-\sum_{j=1}^{d-1}  (e_j^k)_i \partial_{v_j} \varphi_k (s) + (e_d^k)_i\right)\,}{\lambda^{d-1}(s)}{}\nonumber\\
			&= \myint{I_k}{\eta_k \extension{h}{\LipschitzComplement} \varrho g(T_k(s, \varphi_k(s))) \left(-\sum_{j=1}^{d-1}  (e_j^k)_i \partial_{v_j} \varphi_k (s) + (e_d^k)_i\right)\,}{\lambda^{d-1}(s)}{}\nonumber\\
			&= -\myint{\SetA}{\eta_k \extension{h}{\LipschitzComplement} \varrho g (x) \eucscalar{\nu_{\Lipschitz}}{e_i}(x)\,}{H^{d-1}(x)}{}\label{eq: thm calculation of UK boudary term 2 done}
		\end{align}
		for $\extension{h}{\LipschitzComplement}$ defined as in Lemma \ref{lem: hplus and hminus}.
		Putting together Equations \eqref{eq: thm calculation of Uk Zerlegung mit LDC}, \eqref{eq: thm calculation of UK haupt term done}, \eqref{eq: thm calculation of UK boudary term 1 done} and \eqref{eq: thm calculation of UK boudary term 2 done} we have as a result
		\begin{align*}
			&\myint{U_k\setminus\SetA}{\partial_{x_i}g \eta_k h(x) \,\varrho(x)}{\lambda^d(x)}{}  
			= - \myint{U_k\setminus \SetA}{g(x) \partial_{x_i}(\eta_k h \varrho)(x)\,}{\lambda^d(x)}{}\\
			&+ \myint{\SetA}{\eta_k g (x) (\extension{h}{\Lipschitz}-\extension{h}{\LipschitzComplement})(x) \eucscalar{\nu_{\Lipschitz}}{e_i}(x)\,\varrho(x)}{H^{d-1}(x)}{}.
		\end{align*}
	\end{enumerate}
\end{proof}

After this preparation, we are in the position to compute the generator $L^{\varrho}$ of the Dirichlet form $(\Erho, D(\Erho))$ on $\ContDiff{\R}{0}{c} \cap \ContDiff{\R \setminus\SetA}{2}{b} \subseteq L^2(\R^d, \measdens{\mu}{\varrho})$ under Condition \ref{cond: d is one} for $d=1$ and under the Condition \ref{cond: Lipschitz boundary} for $d \geq 2$.

\begin{thm}\label{thm: Calculation Generator d greaterequal two}
	Let $(\Erho, D(\Erho))$ be given as in \eqref{defn: Erho} and assume that Condition \ref{cond: Lipschitz boundary} is satisfied. Let $f \in \ContDiff{\R^d}{0}{c} \cap \ContDiff{\R^d \setminus \SetA}{2}{b}\subseteq \Ltwomurho$.
	Then $f \in D(L^{\varrho})$ and 
	\begin{equation*}
		L^{\varrho}f (x) = \mathbbm{1}_{\R^d\setminus \SetA}(x)\left(\Laplace f + \eucscalar{\nabla f}{\nabla \ln(\varrho)}\right)(x) + \mathbbm{1}_{\SetA}(x)\left(\extension{\left(\partial_{\nu_{\Lipschitz}}f\right)}{\LipschitzComplement}-\extension{\left(\partial_{\nu_{\Lipschitz}}f\right)}{\Lipschitz}\right)(x),
	\end{equation*}
	$ x \in \R^d$ is a $\murho$-version of $L^{\varrho}f \in \Ltwomurho$, where $(L^{\varrho}, D(L^{\varrho}))$ denotes the generator associated to $(\Erho, D(\Erho))$.
\end{thm}

\begin{proof}
	Let $(U_k)_{k \in \N}$ be the locally finite open cover of $\R^d$ and $(\eta_k)_{k \in \N}$ the corresponding partition of unity constructed in Lemma \ref{lem: Partition of unity Rd}.
	We first show that $f \in D(\Erho)$. It holds $f \in L^p(\R^d, \murho)$ for all $1 \leq p \leq \infty$. We fix $i\in \{1, \ldots d\}$.
	For any $g \in \ContDiff{\R^d}{\infty}{c}$ holds with Proposition \ref{prop: Calculation of Uk on d greater two} for $\varrho=\mathbbm{1}$
	\begingroup
	\allowdisplaybreaks
	\begin{align*}
		&\myint{\R^d}{\partial_{x_i}g f\,}{\lambda^d}{} 
		=- \sum_{\substack{k \in \N \\ U_k \cap \SetA = \emptyset}}\myint{U_k}{g\partial_{x_i}\left(\eta_k f\right)\,}{\lambda^d}{}
		- \sum_{\substack{k \in \N \\ U_k \cap \SetA \neq \emptyset}}\myint{U_k\setminus \SetA}{g\partial_{x_i}\left(\eta_k f\right)\,}{\lambda^d}{}\\
		+&\sum_{\substack{k \in \N \\ U_k \cap\SetA \neq \emptyset}}\myint{\SetA}{\eta_k g \left(f^{\text{e}}_{\Lipschitz}-f^{\text{e}}_{\LipschitzComplement}\right)\eucscalar{\nu_{\Lipschitz}}{e_i}\,}{H^{d-1}}{}
		= \myint{\R^d}{g \mathbbm{1}_{\R^d\setminus \SetA}\partial_{x_i}f\,}{\lambda^d}{},
	\end{align*}
	\endgroup
	because $f$ is continuous on $\R^d$ and $(\eta_k)_{k \in \N}$ form a partition of unity. Since $\mathbbm{1}_{\R^d\setminus \SetA}\partial_{x_i}f \in \mathscr{L}^p(\R^d, \lambda^d)$ we conclude $f \in H^{1,p}(\R^d)$ for $1 \leq p \leq \infty$.
	
	Let $(\Psi_{\frac{1}{n}})_{n \in \N}$ be a standard approximate identity and set 
	\begin{align*}
		f_n \vcentcolon \R^d \to \R, 
		x \mapsto \myint{\R^d}{\Psi_{\frac{1}{n}}(x-y)f(y)\,}{\lambda^d(y)}{}
	\end{align*} 
	for all $n \in \N$. 
	Then, $(f_n)_{n \in \N} \subseteq \ContDiff{\R^d}{\infty}{c}$ with $\supp{f_n}, \supp{f} \subseteq K$ for all $n \in\N$ for some $K \subseteq \R^d$ compact . Thus,
	\begin{align}\label{eq: thm Calculation Generator conv for f in DE}
		0 \leq \norm{f-f_n}_{\Ltwomurho}^2 
		\leq C_{\varrho, K}(\sup_{x \in \R^d}\abs{f_n(x)-f(x)})^2  \murho(K) \to 0 \text{ as }n \to \infty
	\end{align}
	where $C_{\varrho, K} \in \R$ is an upper bound of $\varrho$ on $K$. 
	Additionally, for all $x \in \R^d$
	\begin{align*}
		\partial_{x_i}f_n(x)
		= - \myint{\R^d}{\partial_{y_i}\Psi_{\frac{1}{n}}(x-y) f(y)\,}{\lambda^d(y)}{}
		= \myint{\R^d}{\Psi_{\frac{1}{n}}(x-y) \mathbbm{1}_{\R^d\setminus \SetA}(y) \partial_{y_i}f(y)\,}{\lambda^d(y)}{}.
	\end{align*}
	This shows that $\partial_{x_i}f_n$ is a continuous representative of $\Psi_{\frac{1}{n}}\ast \mathbbm{1}_{\R^d \setminus \SetA}\partial_{x_i}f$ for all $n \in \N$. In particular, 
	\begin{align*}
	\lim_{n \to \infty}\norm{\partial_{x_i}f_n-\mathbbm{1}_{\R^d\setminus \SetA}\partial_{x_i}f}_{L^2(\R^d,\lambda^d)} = 0
\end{align*}
	what implies that the limit $\lim_{n,m \to \infty}\myint{\R^d}{(\partial_{x_i} f_n - \partial_{x_i}f_m)^2\,\varrho}{\lambda^d}{}=0$.
	Consequently, $(f_n)_{n \in \N}$ is $\Erho$-Cauchy and with Equation \eqref{eq: thm Calculation Generator conv for f in DE} we conclude $f \in D(\Erho)$.
	
	Let $g \in \ContDiff{\R^d}{\infty}{c}$. 
	We compute applying Proposition \ref{prop: Calculation of Uk on d greater two} 
	\begingroup
	\allowdisplaybreaks
	\begin{align*}
		&\Erho(f,g)
		= \sum_{i=1}^d\left(\sum_{\substack{k \in \N\\ U_k \cap \SetA\\ = \emptyset}}\myint{U_k}{\eta_k \partial_{x_i}f \partial_{x_i}g\,\varrho}{\lambda^d}{} + \sum_{\substack{k \in \N \\ U_k \cap \SetA \neq \emptyset}}\myint{U_k\setminus \SetA}{\eta_k \partial_{x_i}f \partial_{x_i}g\,\varrho}{\lambda^d}{}\right)\\
		&=  -\sum_{i=1}^d\sum_{\substack{k \in \N\\ U_k \cap \SetA = \emptyset}} \myint{U_k}{g\ \partial_{x_i}\left(\eta_k \partial_{x_i}f \varrho\right)\,}{\lambda^d}{}-  \sum_{i=1}^d\sum_{\substack{k \in \N \\ U_k \cap\SetA \neq \emptyset}}\myint{U_k\setminus\SetA}{g\ \partial_{x_i}\left(\eta_k \partial_{x_i}f \varrho\right)\,}{\lambda^d}{}\\&
		+ \sum_{i=1}^d\sum_{\substack{k \in \N \\ U_k \cap \SetA \neq \emptyset}}\myint{\SetA}{\eta_k  g \left(\extension{\left(\partial_{x_i}f\right)}{\Lipschitz}-\extension{\left(\partial_{x_i}f\right)}{\LipschitzComplement}\right)\eucscalar{\nu_{\Lipschitz}}{e_i}\, \varrho}{H^{d-1}}{}\\
		&= -\sum_{i=1}^d \myint{\R^d\setminus \SetA}{g\ \partial_{x_i}\left(\partial_{x_i}f \varrho\right)\,}{\lambda^d}{}
		+ \sum_{i=1}^d\myint{\SetA}{g \left(\extension{\left(\partial_{x_i}f\right)}{\Lipschitz}-\extension{\left(\partial_{x_i}f\right)}{\LipschitzComplement}\right)\eucscalar{\nu_{\Lipschitz}}{e_i}\, \varrho}{H^{d-1}}{}\\
		&=  -\bigg(\myint{\R^d}{\mathbbm{1}_{\R^d\setminus \SetA}\left(\Laplace f + \eucscalar{\nabla f}{\ln(\varrho)}\right) g \, \varrho}{\lambda^d}{}
		+  \myint{\SetA}{\left(\extension{\left(\partial_{\nu_{\Lipschitz}}f\right)}{\LipschitzComplement} - \extension{\left(\partial_{\nu_{\Lipschitz}}f\right)}{\Lipschitz}\right)g\, \varrho}{H^{d-1}}{}\bigg).\\
	\end{align*}
	\endgroup
	We define 
	\begin{equation*}
		L^{\varrho}f (x) \vcentcolon= \mathbbm{1}_{\R^d\setminus \SetA}(x)\left(\Laplace f + \eucscalar{\nabla f}{\ln(\varrho)}\right)(x) + \mathbbm{1}_{\SetA}(x)\left(\extension{\left(\partial_{\nu_{\Lipschitz}}f\right)}{\LipschitzComplement}-\extension{\left(\partial_{\nu_{\Lipschitz}}f\right)}{\Lipschitz}\right)(x)
	\end{equation*}
	for $x \in \R^d$.
	Then, $	L^{\varrho}f \in \Lcaltwomurho$. Since $D\subseteq D(\Erho)$ is dense \withrespectto $\left(\Erho_1\right)^{\frac{1}{2}}$, a standard approximation argument and the fact that $\Erho$ is continuous in the second component yield $\Erho(f,g) = (-L^{\varrho}f, g)_{\Ltwomurho}$ for all $g \in D(\Erho)$.
\end{proof}

For $d=1$, we obtain a similar statement. 
\begin{thm}\label{thm: Calculation Generator d is one}
	Let $(\Erho, D(\Erho))$ be given as in \eqref{defn: Erho} and assume that Condition \ref{cond: d is one} is satisfied. Let $f \in \ContDiff{\R}{0}{c} \cap \ContDiff{\R \setminus\SetA}{2}{b}\subseteq \Ltwomurho$.
	Then $f \in D(L^{\varrho})$ and 
	\begin{equation*}
		L^{\varrho}f (x) = \mathbbm{1}_{\R\setminus \SetA}(x)\left(f'' + f' (\ln(\varrho))'\right)(x) + \mathbbm{1}_{\SetA}(x)\left(\extension{(f')}{r}-\extension{(f')}{l}\right)(x)
	\end{equation*}
	for $x \in \R$ is a $\murho$-version of $L^{\varrho}f \in L^2(\R, \murho)$, where $(L^{\varrho}, D(L^{\varrho}))$ denotes the generator associated to $(\Erho, D(\Erho))$.
\end{thm}

\begin{proof}
	We define $U_z \vcentcolon = (z, z+2)$ for $z \in \Z$ and denote by $(\eta_z)_{z \in \Z}$ the corresponding partition of unity.
	The same arguments as in the proof of Theorem \ref{thm: Calculation Generator d greaterequal two} show that  $f \in D(\Erho)$.
	Let $g \in \ContDiff{\R^d}{\infty}{c}$. 
	We compute applying Proposition \ref{prop: calculation of Uz in $d=1$}
	\begingroup
	\allowdisplaybreaks
	\begin{align*}
		&\Erho(f,g)
		= \sum_{\substack{z \in \Z\\ U_z \cap \SetA = \emptyset}}\myint{U_z}{\eta_z f'g'\,\varrho}{\lambda^d}{} + \sum_{\substack{z \in \Z \\ U_z \cap \SetA \neq \emptyset}}\myint{U_z\setminus \SetA}{\eta_k f'g'\,\varrho}{\lambda}{}
		=  -\sum_{\substack{z \in \Z\\ U_z \cap \SetA = \emptyset}} \myint{U_z}{\left(\eta_z f' \varrho\right)'g\,}{\lambda}{}\\&
		-  \sum_{\substack{z \in \Z \\ U_z \cap \SetA \neq \emptyset}}\myint{U_z\setminus \SetA}{\left(\eta_z f' \varrho\right)'g\,}{\lambda}{}
		+ \sum_{\substack{z \in \Z \\ U_z \cap \SetA \neq \emptyset}}\myint{\SetA}{\eta_z g \left(\extension{\left(f'\right)}{l}-\extension{\left(f'\right)}{r}\right)\, \varrho}{\mS}{}= -\myint{\R\setminus\SetA}{g\ \left( f' \varrho\right)'\,}{\lambda}{}\\&
		+\myint{\SetA}{\left(\extension{\left(f'\right)}{l}-\extension{\left(f'\right)}{r}\right)g\, \varrho}{\mS}{}
		= -\bigg(\myint{\R\setminus\SetA}{\left( f''  + f'\ln(\varrho)'\right)g \,\varrho}{\lambda}{}
		+\myint{\SetA}{\left(\extension{\left(f'\right)}{r}- \extension{\left(f'\right)}{l}\right)g \,\varrho}{\mS}{}\bigg).
	\end{align*}
	\endgroup
	
	We define 
	\begin{equation*}
		L^{\varrho}f (x,y) = \mathbbm{1}_{\R\setminus\SetA}(x)\left(f''  + f'\ln(\varrho)'\right)(x) + \mathbbm{1}_{\SetA}(x)\left(\extension{\left(f'\right)}{r}-\extension{\left(f'\right)}{l}\right)(x).
	\end{equation*}
	Then, $	L^{\varrho}f \in \Lcaltwomurho$. Again, a standard argument shows that for all $g \in D(\Erho)$ it holds $\Erho(f,g) =\Ltwomurhoscalar{-L^{\varrho}f}{g}$.
\end{proof}

\begin{rem}
	For the choice $d=1, A=\{0\}, \mS=\delta_0$ and $\varrho=\mathbbm{1}$, we know from Proposition \ref{prop: domain if $d=1$} that $D(\ErhoOne)=H^{1,2}(\R)$. In the Bachelor's thesis \cite{Viktoria} it is shown that under the these Conditions on $d, \SetA$ and $\mS$ and  the density $\varrho>0$ is infintely often differentiable, the domain of the generator $(L^{\varrho}, D(L^{\varrho}))$ is given by
	\begin{align*}
		D(L^{\varrho})\vcentcolon = \{f \in D(\Erho)\cap H^{2,2}_{\text{loc}}(\R\setminus\{0\})\vcentcolon f^{(2)}+f^{(1)}(\ln(\varrho))' \in L^2(\R, \measdens{\lambda}{\varrho})\}
	\end{align*}
	and $L^{\varrho}f(x)= \mathbbm{1}_{\R\setminus\{0\}}(f^{(2)}+f^{(1)}(\ln(\varrho))')+\mathbbm{1}_{\{0\}}\left(\lim_{\varepsilon \to 0}\left(\widetilde{f}^{(1)}(\varepsilon)-\widetilde{f}^{(1)}(-\varepsilon)\right)\right)$, where $f^{(2)}$ denotes the weak second derivative of $f \in H^{2,2}_{\text{loc}}(\R\setminus\{0\})$ and $\widetilde{f}^{(1)}$ a continuous version of the first weak derivative $f^{(1)}\in H^{1,2}_{\text{loc}}(\R\setminus\{0\})$. However, we will not use this result in the sequel.
\end{rem}

%% file: sections/SectionProcess_DBM_ARXIV.tex
\section{The associated Markov process}

\subsection{Existence}
It is well known that there is a one-to-one correspondence between Markovian semigroups (denoted by $(T_t)_{t \geq 0})$ and Dirichlet forms. If a Dirichlet form is regular, \eg \cite{Fukushima}, Theorem 7.2.1 (p.380) states the existence of an associated Hunt process. We refer to \cite{Fukushima} for the definition of Markov process, Hunt process and diffusion.
The association of regular Dirichlet forms and Hunt processes is established through the corresponding Markovian semigroup. Some properties of the Dirichlet form imply properties of the associated process. In our case, given that $(\Erho, D(\Erho))$ defined in \eqref{defn: Erho} is strongly local (and under suitable conditions conservative), we are able to state existence of an associated Hunt process with continuous sample paths (and infinite lifetime). Before we formulate this existence theorem, we briefly look at some notions from potential theory.

\begin{rem}
	Let $(\mathcal{E}, D(\mathcal{E}))$ be a regular Dirichlet form.
	\begin{enumerate}[(i)]
		\item The definition of capacity can be found in \cite{Fukushima}, Section 2.1 (p.66). This definition corresponds to the one given in \cite{MaRoeckner}, Section III.2 (p.75) for the choice of $h=g=1$, see \cite{MaRoeckner}, Remark III.2.9(ii) and Exercise III.2.10 (p.78). We denote the capacity with respect to $(\mathcal{E}, D(\mathcal{E}))$ by $\text{Cap}_{\mathcal{E}}$.
		\item Let $B$ be a subset of $\R^d$. A statement depending on $x \in B$ is said to hold q.e.~on $B$ if there exists a set $N \subseteq B$ of zero capacity such that the statement is true for every $x \in B \setminus N$ . “\qe” is an abbreviation of “quasi-everywhere”, see \cite{Fukushima}, p.68.
		\item Let $u$ be an extended real valued function defined \qe on $\R^d$. We call u quasi continuous if there exists for any $\eps>0$ an open set $G \subseteq \R^d$ such that $\Capac{\mathcal{E}}{G}<\eps$ and the restriction of $u$ to $\R^d\setminus G$ is continuous, see \cite{Fukushima}, p.69. As in the proof of \cite{Fukushima}, Theorem 2.1.2 (p.69) one can show that if $u$ is quasi continuous, there exists a $\mathcal{E}$-nest $(F_k)_{k \in \N}$ such that $\restrictfunc{u}{F_k}$ is continuous for all $k \in \N$. For the definition of $\mathcal{E}$-nest, see e.g.~\cite{Fukushima}, p.69 where the underlying Dirichlet form is omitted in the notation and it is only referred to as nest. 
	\end{enumerate}
\end{rem}

\begin{thm}\label{thm: Existence of processs}
	There exists a diffusion process, \ie a strong Markov process with continuous sample paths denoted by
	\begin{align*}
		\Markov{\varrho} \vcentcolon = (\Omega, \mathcal{F}, (\mathcal{F}_t)_{t \geq 0}, (X_t)_{t \geq 0}, (\Theta_t)_{t \geq 0}, (P_x)_{x \in \R^d})
	\end{align*}
	with state space $\R^d$ associated to the Dirichlet $(\Erho, D(\Erho))$ defined in \eqref{defn: Erho}, i.e. for all ($\murho$--versions of) $f \in \Ltwomurho$ and all $t>0$ the function
	\begin{align*}
		\R^d \ni x \mapsto p_t f(x) \vcentcolon= \Expec{x}{}{f(X_t)} = \myint{\Omega}{f(X_t)\,}{P_x}{} \in [0, \infty]
	\end{align*}
	is a quasi continuous version of $T_tf$. In particular, $\Markov{\varrho}$ is $\murho$ symmetric, \ie
	\begin{align*}
		\myint{\R^d}{p_t f g\, \varrho}{\mu}{} = \myint{\R^d}{f p_t g\, \varrho}{\mu}{}
	\end{align*}
	for all $f,g \vcentcolon \R^d \to [0, \infty)$ measurable and all $t>0$. The process is up to equivalence (in the sense of \cite{Fukushima}, p.167) unique. Under the conditions of Proposition \ref{prop: E is recurrent}, $\Markov{\varrho}$ is additionally conservative, \ie the process has infinite lifetime. 
\end{thm}

\begin{proof}
	Follows from \cite{Fukushima}, Theorem 7.2.2 (p.380) and \cite{Fukushima}, Exercise 4.5.1 (p.187). Furthermore, \cite{MaRoeckner}, Exercise IV.2.9 (p.100) states that for all ($\murho$--versions of) $f \in \Ltwomurho$ and all $t>0$ the function $p_tf$ is a quasi continuous version of $T_tf$.
\end{proof}
\Wlog we assume that $(\mathcal{F}_t)_{t \geq 0}$ denotes the minimum completed admissible filtration, see \cite{Fukushima}, Theorem A.2.1 (p.389). For $x \in \R^d$, we denote the expectation \withrespectto the probability measure $P_x$ by $\Expec{x}{}{\cdot}$.

In the one-dimensional case, we can strengthen the result and have continuous and not only quasi continuous versions of $T_tf$ for $f \in \Ltwomurho$. To show this, we first prove two general lemmata. For $f \in D(\Erho)$, we denote by $\widetilde{f}$ the continuous $\lambda$-version of $f \in H^{1,2}_{\text{loc}}(\R)$, whose existence follows from Proposition \ref{prop: Domain is subset of Sob Space} and the Sobolev embedding Theorem \cite{AdamsSobSpaces}, Lemma 5.15 (p.107). In the following, we always refer to this lemma whenever a Sobolev embedding theorem is mentioned.
\begin{lem}\label{lem: cap zero}
	Let $(\mathcal{E}, D(\mathcal{E}))$ be any regular Dirichlet form on $L^2(\R, \nu)$ where $\nu$ is a positive Radon measure on $\R$ with supp$(\nu)=\R$. Moreover, let $\lambda$ be absolutely continuous with respect to $\nu$ (in notation $\lambda << \nu$). Furthermore, assume $D(\mathcal{E})\subseteq H^{1,2}_{\text{loc}}(\R)$ in the sense that: If $f \in \mathscr{L}^2(\R, \nu)$ is a $\nu$-version of $f \in D(\mathcal{E})$, then $f \in \mathscr{L}^2_{\text{loc}}(\R, \lambda)$ and the corresponding equivalence class of $f$ \withrespectto $\lambda$ is an element of $H^{1,2}_{\text{loc}}(\R)$. If for all $U \subseteq \R$ open and bounded there exists a constant $C_U$ such that
	\begin{align*}
		\norm{\restrictfunc{f}{U}}_{H^{1,2}(U)}^2 \leq C_U \mathcal{E}_1(f,f)
	\end{align*}
	for all $f \in D(\mathcal{E})$, then the following holds:
	\begin{enumerate}[(i)]
		\item Let $B \subseteq \R$. Then $\Capac{\mathcal{E}}{B}=0$ if and only if $B=\emptyset$. 
		\item Let $x \in \R$ arbitrary. Then there exists an open  and bounded neighbourhood of $x$ denoted by $U_x$ and a constant $C_{x}\in \R_{>0}$ such that
		\begin{align*}
			\Capac{\mathcal{E}}{\{y\}}\geq C_{x}\text{ for all }y \in U_x
		\end{align*}
		i.e. there exists a positive constant that uniformly bounds the Capacity of singleton sets from below for all singletons from $U_x$.
	\end{enumerate}
\end{lem}
\begin{proof}
	We note that since $D(\mathcal{E})\subseteq H^{1,2}_{\text{loc}}(\R)$, each $f \in D(\Erho)$ admits a continuous $\lambda$-version that we denote by $\widetilde{f}$ and which is constructed using a classical Sobolev embedding theorem, see e.g.~\cite{AdamsSobSpaces}, Lemma 5.15 (p.107). The condition $\lambda << \nu$ ensures that two different $\nu$-versions of $f \in D(\mathcal{E})$ are also equivalent \withrespectto $\lambda$, and therefore, the continuous $\lambda$-version of $f$ is unique.\\
	The proof uses the ideas of \cite{Sauerbrey}, Theorem 3.14(i) in Arxiv Version (p.14). Let $x \in \R$. We define
	\begin{equation}
		E_x \vcentcolon D(\mathcal{E}) \to \R, x \mapsto \widetilde{f}(x).
	\end{equation}
	This mapping is well defined and continuous. Indeed, let $U_x$ be an open and bounded neighbourhood of $x$. Then $\restrictfunc{\widetilde{f}}{U_x}= \widetilde{\restrictfunc{f}{U_x}}$, where $\widetilde{\restrictfunc{f}{U_x}}$ denotes the bounded and continuous $\lambda$-version of $\restrictfunc{f}{U_x} \in H^{1,2}(U_x)$ given by the Sobolev embedding theorem mentioned above. Using this Theorem and our assumptions, we have
	\begin{align*}
		&\abs{\widetilde{f}(x)}^2 = \abs{\widetilde{\restrictfunc{f}{U_x}}(x)}^2 \leq \norm{\widetilde{\restrictfunc{f}{U_x}}}_{\text{sup}}^2\leq C_{\text{emb}} \norm{\restrictfunc{f}{U_x}}_{H^{1,2}(U_x)}^2 \leq C_{\text{emb}} C_{U_x} \mathcal{E}_1(f,f),
	\end{align*}
	where $C_{\text{emb}}>0$ is a constant from the Sobolev embedding theorem and $C_{U_x}$ exists by assumption. Since no constant depends on $f$,  $E_x$ is continuous.\\ 
	By the definition of Capacity and \cite{Fukushima}, there exists a sequence of subsets of $\R$ denoted by $(\mathscr{O}_k)_{k \in \N}$ with $x \in  \mathscr{O}_k$ for all $k \in \N$ and a sequence of functions $(f_k)_{k \in \N} \in D(\mathcal{E})$ such that $f_k \geq 1$ $\murho$-a.e.~on $\mathscr{O}_k$ for all $k \in \N$ and $\mathcal{E}_1(f_k, f_k) \to \Capac{\mathcal{E}}{\{x\}}$. In particular, $\widetilde{f_k}(y) \geq 1$ for all $y \in \mathscr{O}_k$ by continuity. This implies 
	\begin{align*}
		1 \leq \abs{\widetilde{f_k}(x)}^2 = \abs{E_x f_k}^2 \leq C_{\text{emb}}C_{U_x}\mathcal{E}_1(f_k,f_k)
	\end{align*}
	what gives $\mathcal{E}_1(f_k,f_k) \geq \frac{1}{C_{\text{emb}}C_{U_x}}$ for all $k \in \N$.
	Therefore,
	\begin{align}\label{ineq: Ineq for Cap of singletons in bounded set}
		\Capac{\mathcal{E}}{\{x\}} \geq \frac{1}{C_{U_x} C_{\text{emb}}}
	\end{align}
	and by a monotonicity argument, every non-empty subset of $\R$ has positive Capacity what shows (a). 
	
	Let $z \in U_x$. Since $U_x$ is an open and bounded neighbourhood of $z$, we obtain with the same arguments and same constants that $\abs{\widetilde{f}(z)}^2 \leq C_{U_x} C_{\text{emb}} \mathcal{E}_1(f,f)$
	for all $f \in D(\mathcal{E})$. Thus, we again have $\Capac{\mathcal{E}}{\{z\}} \geq \frac{1}{C_{U_x} C_{\text{emb}}}.$
	With $C_{x} \vcentcolon = \frac{1}{C_{U_x} C_{\text{emb}}}>0$ we obtain (b).
\end{proof}

The condition in Lemma \ref{lem: cap zero}(b) enables us to show that every quasi continuous function is continuous.
\begin{lem}\label{lem: quasi cont is cont}
	Let $X$ be a locally compact separable metric space, let $m$ be a positive Radon measure on $X$ such that supp($m$)=$X$ and let $(\mathcal{E}, D(\mathcal{E}))$ be a regular Dirichlet form on $L^2(X,m)$. Assume that for all $x \in X$, there exists an open neighbourhood $U_x$ and a constant $C_{x}>0$ such that 
	\begin{align}\label{ineq: Cap locally uniformly bounded from below}
		\Capac{\mathcal{E}}{\{y\}} \geq C_{x}\text{ for all }y \in U_x.
	\end{align}
	Then every quasi continuous function is continuous.
\end{lem}
\begin{proof}
	Let $u$ be a quasi continuous function and $(F_k)_{k \in \N}$ be an $\mathcal{E}$-nest such that $\restrictfunc{u}{F_k}$ is continuous for all $k \in \N$. Let $x \in X$ and $U_x$ an open neighbourhood such that Inequality \eqref{ineq: Cap locally uniformly bounded from below} is fulfilled. Assume that $U_x \cap F_k^{\text{C}} \neq \emptyset$ for all $k \in \N$. Then we find a sequence $(y_k)_{k \in \N}$ with $y_k \in U_x \cap F_k^{\text{C}}$ for all $k \in \N$ and conclude for all $k \in \N$ that
	\begin{equation*}
		\Capac{\mathcal{E}}{F_k^{\text{C}}}\geq \Capac{\mathcal{E}}{\{y_k\}} \geq C_{x}
	\end{equation*}
	what implies
	\begin{equation*}
		\lim_{k \to \infty}\Capac{\mathcal{E}}{F_k^{\text{C}}}\geq C_{x}>0
	\end{equation*}
	which is a contradiction to the fact that $(F_k)_{k \in \N}$ is an $\mathcal{E}$-nest.
	Thus, there exists a $k \in \N$ such that $U_x \subseteq F_k$ and in particular, $x$ is an inner point of $F_k$. By assumption, $\restrictfunc{u}{F_k}$ is continuous and therefore, $u$ is continuous in $x$.
\end{proof}

\begin{lem}\label{lem: Existence of Process in all points 1 dim}
	If $d=1$, then $(\Erho, D(\Erho))$ given in \eqref{defn: Erho} fulfils the Conditions for Lemma \ref{lem: cap zero}(b), see Remark \ref{rem: Cinfcomp is dense in Ltwomurho}, Proposition \ref{prop: Domain is subset of Sob Space} and using the Sobolev embedding theorem.  Thus, using Lemma \ref{lem: quasi cont is cont}, we can replace quasi continuous by continuous in Theorem \ref{thm: Existence of processs} . This shows for any $f \in \LcaltwomurhoOneDim$ and all $t >0$ that $p_tf \in \ContDiff{\R}{0}{}$, what is called $\mathscr{L}^2$- strong Feller property.
\end{lem}
\subsection{Fukushima decomposition and associated SDE}
With the condition
\begin{cond}\label{cond: rho integrable}
	Let $\varrho \in \ContDiff{\R^d}{1}{}\cap \mathscr{L}^1(\R^d, \mu)$, \ie continuously differentiable and integrable \withrespectto $\mu$, and $\varrho>0$
\end{cond}
we show that the process $\Markov{\varrho}$ given in Theorem \ref{thm: Existence of processs} solves
\begin{align}\left\{
	\begin{array}{ll}
		dX^i_t &= \sqrt{2}\cdot\mathbbm{1}_{\mathbb{R}^d\setminus A}\big(X_t\big)\, dB^i_t + \partial_{x_i}\ln(\varrho)(X_t)\cdot\mathbbm{1}_{\mathbb{R}^d\setminus A}\big(X_t\big)\, dt\\
		X^i_0 &= x_i
	\end{array}\right.\text{for}\quad i\in\{1,\ldots,d\},
\end{align}
for 
\begin{enumerate}[(i)]
	\item for all starting points $x \in \R$ with respect to the underlying Dirichlet form $(\Erho, D(\Erho))$ given in \eqref{defn: Erho}  if $\SetA$ fulfils Condition \ref{cond: d is one} and $\varrho$ fulfils Condition \ref{cond: rho ess bounded} or \ref{cond: rho integrable},
	\item quasi every starting point $x=(x_1,\ldots,x_d)^T\in\mathbb{R}^d$ with respect to the underlying Dirichlet form $(\Erho, D(\Erho))$ if $\SetA$ fulfils Condition \ref{cond: Lipschitz boundary} and $\varrho$ fulfils Condition \ref{cond: rho integrable}.
\end{enumerate}
In both cases, the main ingredient for the proof is the Fukushima decomposition, \cite{Fukushima}, Theorem 5.2.5 (p.252).
Some preparatory work is necessary to check the assumptions.
We start with recalling some definition from the context of additive functionals. This requires an underlying regular Dirichlet form. In our case, we define them for $(\Erho, D(\Erho))$.
\begin{definition}\hfill
	\begin{enumerate}[(i)]
		\item{(\cite{Fukushima}, p.83)} A positive Borel measure $\nu$ on $(\R^d, \Borel{\R^d})$ is called smooth, if 
		\begin{description}
			\item[(S1)] $\nu$ charges no sets of capacity zero,
			\item[(S2)] there exists an increasing sequence $\{F_n\}_{n \in \N}$ of closed sets such that $\nu(F_n) < \infty$ for all $n \in \N$ and $\lim_{n \to \infty} \Capac{\Erho}{K \setminus F_n}=0$ for any compact set $K$.
		\end{description}
		The set of all smooth measures is denoted by $\Smoothmeasures$.
		\item {(\cite{Fukushima}, p. 77)} A positive Radon measure $\nu$ is said to be of finite energy integral if there exists a constant $0<C<\infty$ such that
		\begin{equation}
			\myint{\R^d}{\abs{g(x)}}{\nu(x)}{} \leq C \sqrt{\Erho_1(g,g)}
		\end{equation}
		holds true for all $g \in \ContDiff{\R^d}{0}{c}\cap D(\Erho)$. The set of all positive Radon measures of finite energy integral is denoted by $\Smoothmeasures_0$.
	\end{enumerate}
\end{definition}

\begin{rem}[{\cite{Fukushima}, p.77}]
	A positive Radon measure $\nu$ on $\R^d$ is of finite energy integral if and only if for each $\alpha>0$ there exists a unique function $U_{\alpha}\nu \in D(\Erho)$ such that
	\begin{equation*}
		\Erho_{\alpha}(U_{\alpha}\nu, g) = \myint{\R^d}{g(x)}{\nu(x)}{}
	\end{equation*}
	for all $g \in \ContDiff{\R^d}{0}{c}\cap D(\Erho)$. The function $U_{\alpha}\nu$ is called $\alpha$--potential.
\end{rem}

\begin{definition}[{\cite{Fukushima}, p.81}]
	We set
	\begin{equation}
		\Smoothmeasures_{00} \vcentcolon = \{\nu \in \Smoothmeasures_0 \vcentcolon \nu(\R^d) < \infty, \norm{U_{1}\nu}_{L^{\infty}(\R^d, \murho)}< \infty\}.
	\end{equation}
\end{definition}

We refer to \cite{Fukushima}, Section 4.1 (p.152) for the definition of (properly) exceptional. \\
Based on the constructed process $\Markov{\varrho}$, see Theorem \ref{thm: Existence of processs} and $(\Erho, D(\Erho))$, we define additive functionals. Here, $\zeta$ denotes the lifetime of $\Markov{\varrho}$.
\begin{definition}[{\cite{Fukushima}, Section 5.1 (p.222) \& \cite{Fukushima}, p.235-236}]
	A family $(\Additive_t)_{t \geq 0}$ of extended real valued functions $\Additive_t \vcentcolon \Omega \to \R \cup \{-\infty, \infty\}$ is called additive functional (AF in abbreviation) if $\vcentcolon$
	\begin{enumerate}
		\item For each $t\geq 0$, $\omega \mapsto \Additive_t(\omega)$ is $\mathcal{F}_t$ measurable, where $(\mathcal{F}_t)_{t \geq 0}$ denotes the minimum completed admissible filtration.
		\item There exists a set $\Lambda \in \mathcal{F}_{\infty}$ and an exceptional set $N \subseteq \R^d$ such that $P_x(\Lambda)=1$ for all $x \in \R^d\setminus N$, $\Theta_t \Lambda \subseteq \Lambda$ for all $t>0$ and moreover, for each $\omega \in \Lambda$
		\begin{enumerate}[(i)]
			\item $t \mapsto \Additive_t(\omega)$ is right continuous and has the left limit on $[0, \zeta(\omega))$,
			\item $\Additive_0(\omega) =0$,
			\item $\abs{\Additive_t(\omega)}< \infty$ for all $t < \zeta(\omega)$,
			\item $\Additive _t(\omega) = \Additive_{\zeta(\omega)}(\omega)$ for all $t \geq \zeta(\omega)$
			\item $\Additive_{t+s}(\omega)= \Additive_t(\Theta_s(\omega))$ for all $t,s \geq 0$.
		\end{enumerate}
	\end{enumerate}
	The set $\Lambda$ is called defining set, the set $N$ is called exceptional set of the AF $(\Additive_t)_{t \geq 0}$. \\
	Two AFs $\Additive^{(1)}$ and $\Additive^{(2)}$ are equivalent, in notation $\Additive^{(1)} = \Additive^{(2)}$, if there exists a common defining set $\Lambda$ and a common properly exceptional set $N$ such that $P_x(\Lambda)=1$ for all $x \in \R^d\setminus N$ and $\Additive^{(1)}(\omega) = \Additive^{(2)}(\omega)$ for all $t>0$ and all $\omega \in \Lambda$. 
	An AF is finite \resp continuous if for all $\omega \in \Lambda$, $\abs{\Additive_t(\omega)}<\infty$ for all $t\geq0$ \resp $[0, \infty) \ni t \mapsto \Additive_t(\omega) \in \R \cup \{-\infty, \infty\}$ is continuous.
	A $[0, \infty]$-- valued continuous AF is called a positive continuous AF (PCAF in abbreviation). The set of all PCAFs is denoted by $A_c^+$.
	An AF is called cadlag if $t \mapsto \Additive_t(\omega)$ is right continuous and possesses the left limit on $[0, \infty)$ for any $\omega \in \Lambda$.\\
	An AF in the strict sense is an additive functional with exceptional set $N=\emptyset$, i.e.~an additive funcitonal with a defining set $\Lambda \in \mathcal{F}_{\infty}$ such that $P_x(\Lambda)=1$ for all $x \in \R^d$. Two strict AFs $\Additive^{(1)}$ and $\Additive^{(2)}$ are equivalent, in notation $\Additive^{(1)}= \Additive^{(2)}$, if  there exists a common defining set $\Lambda \in \mathcal{F}_{\infty}$ with $P_x(\Lambda)=1$ for all $x \in \R$ such that $\Additive^{(1)}_t(\omega)=\Additive^{(2)}_t(\omega)$ for all $\omega \in \Lambda$ and all  $t \geq 0$. 
\end{definition}

\begin{definition}[{\cite{Fukushima}, Theorem 5.1.3(iii) (p.229)}]
	A measure $\nu \in \Smoothmeasures$ and an AF $\Additive \in A_c^+$ are said to be in Revuz correspondence if for all non--negative $f,h \in \Borel{\R^d}$ it holds
	\begin{align*}
		\myint{\R^d}{\Expec{x}{}{\myint{0}{f(X_s)\,}{\Additive_s}{t}}h(x)\,\varrho(x)}{\mu(x)}{} = \myint{[0,t]}{\myint{\R^d}{p_sh(x)f(x)}{\nu}{}}{\lambda(s)}{}.
	\end{align*}
\end{definition}
Theorem 5.1.4. in \cite{Fukushima}, p.230 states that the family of all equivalence classes of $A_c^+$ and the family $\Smoothmeasures$ are in one-to-one correspondence under the Revuz correspondence.

For a measure $\chi\murho$ with $\chi$ non-negative, bounded and measurable, we derive the corresponding PCAF using that $\murho$ is symmetric.

\begin{prop}\label{prop: Revuz correspondance and S00}
	Let $\chi, \widehat{\chi} $ be non-negative, bounded and measurable. Then 
	\begin{enumerate}[(i)]
		\item $(\myint{[0,t]}{\chi(X_s)}{\lambda(s)}{})_{t \geq 0}$ is a PCAF and in Revuz correspondance with the measure $\chi \murho$. \\
		In particular, if $\chi= \widehat{\chi}\ \murho$-a.e., then $(\myint{[0,t]}{\chi(X_s)}{\lambda(s)}{})_{t \geq 0} = (\myint{[0,t]}{\widehat{\chi}(X_s)}{\lambda(s)}{})_{t \geq 0}$.
		\item If $\varrho$ fulfils  \ref{cond: rho integrable}, then $\chi \murho\in \Smoothmeasures_{00}$.
		\item If $d=1$ and $\varrho$ fulfils \ref{cond: rho ess bounded}, we have that $\mathbbm{1}_G\chi\murho\in\Smoothmeasures_{00}$ for all open and relatively compact $G \subseteq \R$, i.e.~for all open and bounded $G \subseteq \R$. 
	\end{enumerate}
\end{prop}
\begin{proof}
	\begin{enumerate}[(i)]
		\item Note that for each $t \geq 0$ and $\omega \in \Omega$, the integral $\myint{[0,t]}{\chi(X_s(\omega))}{\lambda(s)}{}$ exists and since $X$ is adapted, Tonelli yields that $\Omega \ni \omega \to \myint{[0,t]}{\chi(X_s(\omega))}{\lambda(s)}{} \in \R$ is $\mathcal{F}_t$-measurable. Furthermore, $(\myint{[0,t]}{\chi(X_s)}{\lambda(s)}{})_{t \geq 0}$ is a PCAF. Let $f,h \in \Borel{\R^d}$ and non-negative. Then, applying Tonelli's theorem and the fact that $\murho$ is symmetric, compare Theorem \ref{thm: Existence of processs}, we obtain
		\begin{align*}
			&\myint{\R^d}{\Expec{x}{}{\myint{[0,t]}{f(X_s)\chi(X_s)\,}{\lambda(s)}{}}h(x)\,\varrho(x)}{\mu(x)}{}
			=\myint{[0,t]}{\myint{\R^d}{fp_sh(x)\,\chi(x)\varrho(x)}{\mu(x)}{}}{\lambda(s)}{}
		\end{align*}
		what proves the Revuz correspondence.
		The second statement is clear since $\chi\murho=\widehat{\chi}\murho$ and the corresponding PCAF is up to equivalence unique.
		\item It holds that $\murho$ is a smooth measure. Let $\alpha>0$. 
		Applying \cite{MaRoeckner}, Lemma I.2.12 (p.21) to the sequence $(\varphi_n)_{n \in \N}$ constructed in Proposition \ref{prop: E is recurrent} and taking Condition \ref{cond: rho integrable} into account, we obtain $\mathbbm{1} \in D(\Erho)$. Therefore, also $\frac{1}{\alpha}\mathbbm{1} \in D(\Erho)$. It holds for $g \in \ContDiff{\R^d}{0}{c}\cap D(\Erho)$
		\begin{align*}
			\myint{\R^d}{g(x) \,\varrho(x)}{\mu(x)}{} = \Erho_{\alpha}\left(\frac{1}{\alpha}\mathbbm{1}, g\right).
		\end{align*}
		The uniqueness implies $U_{\alpha}\left(\murho\right) = \frac{1}{\alpha}\mathbbm{1}$ and that $\murho \in \Smoothmeasures_0$. Since $\norm{\frac{1}{\alpha}\mathbbm{1}}_{L^{\infty}(\murho)}< \infty$ and $\murho$ is finite, we conclude $\murho \in \Smoothmeasures_{00}$. \cite{Fukushima}, Lemma 5.1. (p.228) now yields that also $\chi\murho \in \Smoothmeasures_{00}$.
		\item Let $G \subseteq \R$ be open and bounded. We have that $\mathbbm{1}_G\murho$ is a smooth measure. Let $g \in \ContDiff{\R^d}{0}{c}\cap D(\Erho)$. With Proposition \ref{prop: Domain is subset of Sob Space} we conclude that $g \in H^{1,2}(G)$. Using that $\murho$ is $\sigma$-finite, that $\varrho>0$ is continuous, the Sobolev embedding theorem and an equivalent norm on $H^{1,2}(G)$, we have for some constant $ C >0$
		\begin{align*}
			\myint{\R^d}{\abs{g(x)}\, \mathbbm{1}_G\varrho}{\mu(x)}{} \leq \murho(G) \sup_{x \in G}\abs{g(x)} 
			\leq C \sqrt{\Erho_1(g,g)}.
		\end{align*}
		Thus, there exists the $1$-potential $U_1 \vcentcolon = U_1\mathbbm{1}_G\murho$. Furthermore, by \cite{Fukushima}, Lemma 2.2.5 (p.81), there is an increasing sequence of compact sets $(F_n)_{n \in \N}$ such that $\mathbbm{1}_{F_n}\murho \in \Smoothmeasures_{00}$ for all $n \in \N$ and $\Capac{\Erho}{K \setminus F_n} \to 0$ as $n \to \infty$ for any compact set  $K$. Assume the existence of some $x \in \R\setminus \cup_{n \in \N}F_n$ and choose a compact set $K$ that contains $x$. Then $\Capac{\Erho}{K\setminus F_n} \geq \Capac{\Erho}{\{x\}}>0$ by Lemma \ref{lem: cap zero}(b) for all $n \in \N$, which is a contradiction. Thus, there is some $N_0 \in \N$ with $G \subseteq F_{N_0}$ and since $\mathbbm{1}_{F_{N_0}}\mathbbm{1}_G\varrho\mu= \mathbbm{1}_{G}\varrho\mu$, we conclude $ \mathbbm{1}_{G}\varrho\mu \in \Smoothmeasures_{00}$.
	\end{enumerate}
\end{proof}

\begin{rem}[{\cite{Fukushima}, Section 3.2 (p.123)}]
	Let $f \in D(\Erho)$ be essentially bounded. Then there exists a unique positive Radon measure $\nu_{\langle f \rangle}$ satisfying
	\begin{align*}
		\myint{\R^d}{g\,}{\nu_{\langle f \rangle}}{} = 2 \Erho(fg, f) - \Erho(f^2, g)
	\end{align*}
	for all $g \in \ContDiff{\R^d}{0}{c}\cap D(\Erho)$. We call $\nu_{\langle f \rangle}$ the energy measure of $f$.
\end{rem}
To compute the energy measure of $f \in \ContDiff{\R^d}{0}{c}\cap \ContDiff{\R^d \setminus \SetA}{2}{b}\subseteq \Ltwomurho$, we assume the same conditions on $\mS$ and $\SetA$ as in Theorems \ref{thm: Calculation Generator d is one} for $d=1$ or \ref{thm: Calculation Generator d greaterequal two} for $d \geq 2$. This ensures that $f \in D(\Erho)$.
\begin{prop}\label{prop: energy measure}
	Let $(\Erho, D(\Erho))$ be given as in \eqref{defn: Erho} and assume that Condition \ref{cond: d is one} or \ref{cond: Lipschitz boundary} is satisfied. Let $f \in \ContDiff{\R^d}{0}{c}\cap \ContDiff{\R^d \setminus \SetA}{2}{b}\subseteq \Ltwomurho$.  Then the energy measure $\nu_{\langle f \rangle}$ is given by 
	\begin{align*}
		\nu_{\langle f \rangle} =  \measdens{\lambda^d}{\left(2 \sum_{i=1}^{d} \mathbbm{1}_{\R^d\setminus\SetA}(\partial_{x_i} f)^2 \varrho\right)}.
	\end{align*}
\end{prop}
\begin{proof}
	Let $f \in \ContDiff{\R^d}{0}{c}\cap \ContDiff{\R^d \setminus\SetA}{2}{b}\subseteq \Ltwomurho$ and $g \in\ContDiff{\R^d}{0}{c} \cap D(\Erho)$. Then, $fg, f^2 \in D(\Erho)$ and
	\begin{align}\label{eq: Find energy measure calculation 1}
		2\Erho(fg, f) - \Erho(f^2, g) = 2 \sum_{i=1}^d\myint{\R^d}{(fg)^{(i)}f^{(i)}\,\varrho}{\lambda^d}{} - \sum_{i=1}^d\myint{\R^d}{(f^2)^{(i)}g^{(i)}\,\varrho}{\lambda^d}{}.
	\end{align}
	By Proposition \ref{prop: Domain is subset of Sob Space}, $f,g \in \localSobolev{1,2}{\R^d}$ and since both $f$ and $g$ have representatives with compact support in $\R^d$, we even have $f,g \in \Sobolev{1,2}{\R^d}$. Thus $fg, f^2 \in \Sobolev{1,1}{\R^d}$, the derivatives are given by the product rule and Equation \eqref{eq: Find energy measure calculation 1} rewrites to
	\begin{align*}
		&2\Erho(fg, f) - \Erho(f^2, g) = 2 \sum_{i=1}^d\myint{\R^d}{f^{(i)}gf^{(i)}\,\varrho}{\lambda^d}{} + 2 \sum_{i=1}^d\myint{\R^d}{f g^{(i)}f^{(i)}\,\varrho}{\lambda^d}{}\nonumber\\
		-&2\sum_{i=1}^d\myint{\R^d}{f f^{(i)}g^{(i)}\,\varrho}{\lambda^d}{}
		= 2 \sum_{i=1}^{d}\myint{\R^d}{\left(f^{(i)}\right)^2g\,\varrho}{\lambda^d}{}\nonumber
		=  \myint{\R^d}{ g \left(2 \sum_{i=1}^d\left(f^{(i)}\right)^2\right)\,\varrho}{\lambda^d}{}.
	\end{align*}
	The measure $\measdens{\lambda^d}{\left(2 \sum_{i=1}^d\left(f^{(i)}\right)^2\varrho\right)}$ is a positive Radon measure since it is a Borel measure and finite on compact sets. The proof of Theorems \ref{thm: Calculation Generator d is one} under Condition \ref{cond: d is one} and \ref{thm: Calculation Generator d greaterequal two} under Condition \ref{cond: Lipschitz boundary} show that $f^{(i)} = \mathbbm{1}_{\R^d\setminus \SetA}(\partial_{x_i} f)$ for all $i=1, \ldots d$.
	\noindent The uniqueness of the energy measure yields 
	$\nu_{\langle f \rangle} = \measdens{\lambda^d}{\left(2 \sum_{i=1}^d\left(\mathbbm{1}_{\R^d \setminus \SetA}\partial_{x_i} f\right)^2\varrho\right)}$.
\end{proof}

We need to specify some sets of additive functionals.
\begin{rem}[{\cite{Fukushima}, Section 5.2 (p.241ff)}]\hfill
	\begin{enumerate}[(i)]
		\item For an additive functional $(\Additive_t)_{t \geq 0}$ we set the energy of $\Additive$ to be
		\begin{align*}
			e(\Additive) = \lim_{t \downarrow 0}\frac{1}{2t}\myint{\R^d}{\Expec{x}{}{\Additive_t^2}\,\varrho(x)}{\mu (x)}{}
		\end{align*}
		whenever this limit exists.
		\item We define 
		\begin{align*}
			\mathcal{M}=\{&M \vcentcolon M\text{ is a finite cadlag AF such that for each }t>0\ \Expec{x}{}{M_t^2}< \infty\text{ and }\\&\Expec{x}{}{M_t}=0\text{ for \qe }x\in \R^d\}.
		\end{align*}
		The energy of $M \in \mathcal{M}$ exists in $[0, \infty]$ and we set
		\begin{align*}
			\mathcal{\overset{\circ}{M}}=\{M \in \mathcal{M}\vcentcolon e(M) < \infty\}.
		\end{align*}
		An element $M \in \mathcal{M}$ is called martingale additive functional, in abbreviation MAF. An element $M \in \mathcal{\overset{\circ}{M}}$ is called martingale additive functional of finite energy. Indeed, \cite{Fukushima}, Appendix A.3 (p.415 ff)  shows that $M \in \mathcal{M}$ is a square integrable martingale on $(\Omega, (\mathcal{F}_t)_{t \geq 0}, P_x)$ for \qe $x \in \R^d$. Furthermore, it is possible to assign a square bracket $\langle M \rangle$ to $M$ (also called quadratic variation associated to $M$), independent of $x$ for \qe $x \in \R^d$. $\langle M \rangle$ is a PCAF and is therefore in Revuz correspondence with a measure in $\Smoothmeasures$ that we denote by $\nu_{\langle M \rangle}$.
		\item Additionally, we define
		\begin{align*}
			\mathcal{N}_c\vcentcolon = \{R \vcentcolon R \text{ is a finite continuous AF } e(R)=0,\ \Expec{x}{}{\abs{R_t}}< \infty\text{ q.e.~for each }t>0\}.
		\end{align*}
	\end{enumerate}	
\end{rem}

Given this preparation, we formulate the Fukushima decomposition of the process $\Markov{\varrho}$ given in Theorem \ref{thm: Existence of processs}.
\begin{thm}\label{thm: Fuk decomp in 1d}
	Let $(\Erho, D(\Erho))$ be given as in Equation \eqref{defn: Erho} and $\Markov{\varrho}$ the associated process given in Theorem \ref{thm: Existence of processs}. Assume that Condition \ref{cond: d is one} holds and that $\varrho$ fulfils Condition \ref{cond: rho ess bounded} or \ref{cond: rho integrable}. 
	Let $f \in \ContDiff{\R}{0}{c}\cap\ContDiff{\R\setminus\SetA}{2}{b}\subseteq\LtwomurhoOneDim$. Then
	\begin{align}\label{eq: Fukushima decomposition 1 dim}
		f(X_t)-f(X_0)= M_t^{[f]}+ \myint{[0,t]}{L^{\varrho}f(X_s)\, }{\lambda(s)}{} \text{ for all }t \geq 0\ P_x \text{-a.s.~for all } x \in \R,
	\end{align} 
	where $M^{[f]}$ is an MAF in the strict sense whose square bracket in the strict sense equals (in the sense of equivalence of strict AFs)	
	\begin{align*}
		(\langle M^{[f]}\rangle_t)_{t \geq 0} = \left(2 \sum_{i=1}^d\myint{[0,t]}{\mathbbm{1}_{\R^d\setminus \SetA}(X_s)(\partial_{x_i} f(X_s))^2 \,}{\lambda(s)}{}\right)_{t \geq 0}.
	\end{align*}
\end{thm}
\begin{proof}
	We need to check the assumptions of \cite{Fukushima}, Theorem 5.2.5 (p.252). 
	
	Let $t>0, x \in \R$ and $N \in \Borel{\R}$ with $\murho(N)=0$. We have $p_t(x,N)= \myint{R}{\mathbbm{1}_N(X_t)}{P_x}{}$. Since $\mathbbm{1}_N \in \LcaltwomurhoOneDim$, Theorem \ref{thm: Existence of processs} implies with Remark \ref{lem: Existence of Process in all points 1 dim} that $\R \ni x \mapsto p_t\mathbbm{1}_N(x)$ is a continuous $\murho$-version of $T_t\mathbbm{1}_N =0$ in $\LtwomurhoOneDim$. As a consequence, $p_t\mathbbm{1}_N(x)=0$ and the measure $p_t(x, \cdot)$ is absolutely continuous with respect to $\murho$, what shows the \textit{absolute continuity condition}, see \cite{Fukushima}, label (4.2.9) (p. 165). Furthermore, by Proposition \ref{prop: energy measure}, $\mu_{\langle f \rangle}= \measdens{\lambda}{\left(2 \mathbbm{1}_{\R\setminus\SetA}(f')^2 \varrho\right)}=\measdens{\murho}{\mathbbm{1}_{\R\setminus A}\left(2 \mathbbm{1}_{\R\setminus\SetA}(f')^2\right)}$. Assume that $\varrho$ fulfils Condition \ref{cond: rho ess bounded}. Since $f$ has compact support, there exists some $G\subseteq \R$ open and bounded with supp$(f) \subseteq G$. Thus, $\mu_{\langle f \rangle}=  \measdens{\murho}{\mathbbm{1}_G\left(2 \mathbbm{1}_{\R\setminus\SetA}(f')^2 \right)}$ is an element of $\Smoothmeasures_{00}$ by Proposition \ref{prop: Revuz correspondance and S00}(c). By Theorem \ref{thm: Calculation Generator d is one}, there exists a bounded $\murho$-version of $L^{\varrho}f$ that has compact support.  We conclude that $\measdens{\mu}{(L^{\varrho}f)^+\varrho}$ and $\measdens{\mu}{(L^{\varrho}f)^-\varrho}$ are in $\Smoothmeasures_{00}$. If $\varrho$ fulfils Condition \ref{cond: rho integrable}, we can use Proposition \ref{prop: Revuz correspondance and S00}(b) to obtain that $\mu_{\langle f \rangle}, \measdens{\mu}{(L^{\varrho}f)^+\varrho}$ and  $\measdens{\mu}{(L^{\varrho}f)^-\varrho}$ are elements of $\Smoothmeasures_{00}$.
	Now, \cite{Fukushima}, Theorem 5.2.5 (p.252) yields that the finite AF $A^{[f]}$ defined by $A^{[f]}_t\vcentcolon= f(X_{t})-f(X_0), t\geq 0$ admits the decomposition $A^{[f]}=M^{[f]}+R^{[f]}$ where $M^{[f]}$ is an MAF in the strict sense whose square bracket in the strict sense is in Revuz correspondance with $\mu_{\langle f \rangle}$ and thus
	\begin{align*}
		(\langle M^{[f]}\rangle_t)_{t \geq 0} = \left(2 \sum_{i=1}^d\myint{[0,t]}{\mathbbm{1}_{\R^d\setminus \SetA}(X_s)(\partial_{x_i} f(X_s))^2 \,}{\lambda(s)}{}\right)_{t \geq 0}
	\end{align*}
	by Proposition \ref{prop: Revuz correspondance and S00}(a),
	and that there exists some defining set $\Lambda \in \mathcal{F}_{\infty}$ with $P_x(\Lambda)=1$ for all $x \in \R$ such that
	\begin{align*}
		R^{[f]}_t(\omega)=
		\myint{[0,t]}{L^{\varrho}f(X_s(\omega))}{\lambda(s)}{}.
	\end{align*}
	for $\omega \in \Lambda$ and $t \geq 0$. Here,  $L^{\varrho}f$ denotes a bounded $\murho$-version of $L^{\varrho}f \in \LtwomurhoOneDim$.
\end{proof}

\begin{rem}
	Using the $\mathscr{L}^2$-strong Feller property of $(p_t)_{t \geq 0}$, one can even prove for all $f \in D(L)$ that $(f(X_t)-f(X_0)-\myint{[0,t]}{L^{\varrho}f}{\lambda(s)}{})_{t \geq 0}$ is a $P_x$-martingale for all $x \in \R$ .
\end{rem}

A similar statement to the above theorem, but in a weaker form, holds in higher dimensions if $\SetA$ is chosen as Lipschitz boundary (such that we have an explicit representation of the generator) and if $\varrho$ is integrable with respect to $\mu$. This is due to the lack of the absolute continuity condition. As a consequence, we obtain a decomposition of the process $(f(X_t)-f(X_0))_{t \geq 0}$ only for q.e. $x \in \R^d$. 
\begin{thm}\label{thm: Fukushima decomp}
	Let $(\Erho, D(\Erho))$ be given as in Equation \eqref{defn: Erho} and $\Markov{\varrho}$ the associated process given in Theorem \ref{thm: Existence of processs}. Assume that Condition \ref{cond: Lipschitz boundary} is satisfied and that Condition \ref{cond: rho integrable} holds true. Let $f \in \ContDiff{\R^d}{0}{c}\cap\ContDiff{\R^d\setminus\SetA}{2}{b}\subseteq\Ltwomurho$. Then
	\begin{align}\label{eq: Fukushima decomposition}
		f(X_t)-f(X_0)= M_t^{[f]}+ \myint{[0,t]}{L^{\varrho}f(X_s)\, }{\lambda(s)}{} \text{ for all }t \geq 0\ P_x\text{-a.s.~for q.e.} x \in \R^d
	\end{align} 
	where $M^{[f]} \in \mathcal{\overset{\circ}{M}}$, whose square bracket equals (in the sense of equivalence of AFs)	
	\begin{align*}
		(\langle M^{[f]}\rangle_t)_{t \geq 0} = \left(2 \sum_{i=1}^d\myint{[0,t]}{\mathbbm{1}_{\R^d\setminus \SetA}(X_s)(\partial_{x_i} f(X_s))^2 \,}{\lambda(s)}{}\right)_{t \geq 0}.
	\end{align*}
\end{thm}

\begin{proof}
	Again, the idea is to use \cite{Fukushima}, Theorem 5.2.5 (p.252). Unfortunately, we don't know if the \textit{absolute continuity condition} is fulfiled under these conditions. Even if we can not apply this theorem directly, we can use ideas from its proof.\\
	Let $A^{[f]}_t \vcentcolon = f(X_t)-f(X_0)$ for all $t \geq 0$. Note that this defines a finite AF, see \cite{Fukushima}, p. 242. \cite{Fukushima}, Theorem 5.2.2 (p. 247) shows the existence of the decomposition $A^{[f]} = M^{[f]} + R^{[f]}$
	for some $M^{[f]} \in \mathcal{\overset{\circ}{M}}$ and $R^{[f]}\in \mathcal{N}_c$. Further, \cite{Fukushima}, Theorem 5.2.3 (p.250) shows that the energy measure of $f$ equals the measure that is in Revuz correspondence with $\langle M^{[f]}\rangle$, thus identifies $\nu_{\langle M^{[f]}\rangle} = \nu_{\langle f \rangle}$. The Propositions \ref{prop: energy measure} and \ref{prop: Revuz correspondance and S00}(a) show
	\begin{align*}
		(\langle M^{[f]}\rangle_t)_{t \geq 0} = \left(2 \sum_{i=1}^d\myint{[0,t]}{\mathbbm{1}_{\R^d\setminus\SetA}(X_s)(\partial_{x_i} f(X_s))^2 \,}{\lambda(s)}{}\right)_{t \geq 0}.
	\end{align*}
	Furthermore, we have
	\begin{align*}
		\Erho(f,g)= \myint{\R^d}{(-L^{\varrho}f)g \,\varrho}{\mu}{} 
	\end{align*}
	for all $g \in D(\Erho)$, where $L^{\varrho}f$ is defined as in Theorem \ref{thm: Calculation Generator d greaterequal two}.
	In particular, we know by the cited theorems that $L^{\varrho}f$ has a bounded $\murho$-version and we conclude that the positive and negative part of this version $(-L^{\varrho}f)^+$ and $(-L^{\varrho}f)^-$ are non--negative bounded Borel functions. Thus, Proposition \ref{prop: Revuz correspondance and S00}(b) shows that $\measdens{\murho}{(-L^{\varrho}f)^+}$ and $\measdens{\murho}{(-L^{\varrho}f)^-}$  are elements of $\Smoothmeasures_{00}$ and part (a) of this Proposition shows that their corresponding PCAFs denoted by $(R^+_t)_{t \geq 0}$ and $(R^-_t)_{t \geq 0}$, respectively, are given by
	\begin{align*}
		(R^\pm_t)_{t \geq 0} = \left(\myint{[0,t]}{(-L^{\varrho}f)^\pm(X_s)\,}{\lambda(s)}{}\right)_{t \geq 0} 
	\end{align*}
	For all $\omega \in \Lambda^+ \cap \Lambda^-$, where  $ \Lambda^+,  \Lambda^-$, respectively, is the defining set of $R^+, R^-$, respectively, we set
	\begin{align*}
		R_t(\omega) \vcentcolon = -R^+_t(\omega)+R^-_t(\omega) = \myint{[0,t]}{L^{\varrho}f(X_s(\omega))\,}{\lambda(s)}{}
	\end{align*} 
	for $t \geq 0$. $R$  is then an additive functional and by \cite{Fukushima}, p. 244 an element of $\mathcal{N}_c$. It remains to show $R=R^{[f]}$. 
	Let $v \in D(\Erho)$. Applying \cite{Fukushima}, Theorem 5.1.3(vi) (p. 229) to $v^+$ and $v^-$ we have
	\begin{align*}
		\lim_{t \downarrow 0} \frac{1}{t}\myint{\R^d}{\Expec{x}{}{R_t^\pm}v \,\varrho}{\mu}{} = \myint{\R^d}{(-L^{\varrho}f)^\pm v \,\varrho}{\mu}{} 
	\end{align*}
	what yields
	\begin{align*}
		\lim_{t \downarrow 0} \frac{1}{t}\myint{\R^d}{\Expec{x}{}{R_t}v \,\varrho}{\mu}{} = - \Erho(f,v).
	\end{align*}
	By \cite{Fukushima}, Theorem 5.2.4 (p.251) this is equivalent to $R=R^{[f]}$.
\end{proof}

\begin{thm}\label{thm: Process is Bm}
	Assume the same conditions as in Theorem \ref{thm: Fukushima decomp}.
	For all $x \in \R^d \setminus N$, $N$ exceptional,  there exists a $d$-dimensional Brownian motion $B^x$ on a propably enlarged probability space $(\Omega, (\mathcal{F}_t)_{t \geq 0}, P_x)$ such that $P_x$-a.~s.
	\begin{align*}
		X_t^i - X_0^i = \sqrt{2} \myint{0}{\mathbbm{1}_{\R^d\setminus\SetA}(X_s)\,}{({B^x})_s^i}{t} + \myint{0}{\mathbbm{1}_{\R^d\setminus\SetA}(X_s)\partial_{x_i}(\ln \varrho(X_s))\,}{s}{t}
	\end{align*}
	for all $t \geq 0$ and for all $1 \leq i\leq d$. The last integrand has to be read as Lebesgue Integral when $\omega \in \Omega$ is fixed. We deviate from our usual notation to adapt to the common one.
\end{thm}
\begin{proof}
	This proof follows the line of \cite{FGV2016}, Corollary 4.18 (p.751). 
	We fix $i \in \{1, \ldots d\}$. For $k \in \N$ define $\tau_k \vcentcolon = \inf\{t \geq 0 \vcentcolon X_t \notin \overline{B_k(0)}\}$ and a twice continuously differentiable function $\pi_i^k\vcentcolon \R^d \to \R$ with the property that 
	$\pi_i^k(x) = \begin{cases}
		x_i, &\text{ if }\eucnorm{x}< k\\
		0, &\text{ if }x \in \R^d \setminus B_{k+1}(0).
	\end{cases}$\\
	Then, $(\tau_k)_{k \in \N}$ defines a sequence of increasing stopping times with $\tau_k \uparrow \infty$ for $k \to \infty$ and $\pi_i^k \in C_b^2(\R^d)$ for all $k \in \N$. 
	We find  $\Lambda_k \in \mathcal{F}_{\infty}$ and an exceptional set $N_k$ such that $P_x(\Lambda_k)=1$ for all $x \in \R^d \setminus N_k$ and 
	\begin{align*}
		\pi_i^k(X_t(\omega)) - \pi_i^k(X_0(\omega))= M^{[\pi_i^k]}_t(\omega) + \myint{[0,t]}{L^{\varrho}\pi_i^k(X_s(\omega))\,}{\lambda(s)}{}
	\end{align*}
	as well as
	\begin{align*}
		\langle M^{[\pi_i^k]} \rangle_t(\omega) = 2 \sum_{j=1}^d\myint{[0,t]}{\mathbbm{1}_{\R^d\setminus \SetA}(X_s(\omega)) (\partial_{x_j} \pi_i^k(X_s(\omega)))^2\,}{\lambda(s)}{}
	\end{align*}
	for all $t\geq 0$ and for all $\omega \in \Lambda_k$. In particular, we have
	\begin{align}\label{eq: stopped process}
		\pi_i^k(X_{t \wedge \tau_k(\omega)}(\omega)) - \pi_i^k(X_0(\omega))= M^{[\pi_i^k]}_{t \wedge \tau_k(\omega)}(\omega) + \myint{[0,t \wedge \tau_k(\omega)]}{\partial_{x_i}(\ln(\varrho(X_s(\omega)))\,}{\lambda(s)}{}
	\end{align}
	for all $t \geq 0$ and all $\omega \in \Lambda_k$. Since $M^{[\pi_i^k]}$ is a martingale and $t \wedge \tau_k$ is a bounded stopping time for all $k \in \N$ and the filtration is right continuous, $M_{t \wedge \tau_k}^{[\pi_i^k]}$ is a martingale. We denote by $\Omega_c \vcentcolon = \{\omega \in \Omega \vcentcolon t \mapsto X_t(\omega)\text{ is continuous}\}$. We define for $\omega \in \Omega$ and $t \geq0$
	\begin{align*}
		M_t^{[\pi_i]}(\omega) = 
		\begin{cases}
			M_{t \wedge \tau_k(\omega)}^{[\pi_i^k]}(\omega) \text{ for }k\text{ such that }\tau_k(\omega)\geq t, &\text{if }\omega \in \cap_{k \in \N}\Lambda_k \cap \Omega_c \\
			0, &\text{ else.}
		\end{cases}
	\end{align*} 
	Then, $M^{[\pi_i]}$ is a continuous local martingale with $(\tau_k)_{k \in \N}$ as localizing sequence. 
	For $i,j \in \{1,\ldots, d\}$ the polarization identity yields
	\begin{align*}
		\langle M^{[\pi_i]}, M^{[\pi_j]}\rangle_t(\omega) = 2 \delta_{ij} \mathbbm{1}_{\cap_{k \in \N}\Lambda_k \cap \Omega_c}(\omega)\myint{[0,t]}{\mathbbm{1}_{\R^d\setminus\SetA}(X_s(\omega))\,}{\lambda^d(s)}{}.
	\end{align*}
	Finally, \cite{Kallenberg2001}, Theorem 18.12 (p. 358) states that for all $x \in \R^d\setminus \cup_{k \in \N}N_k$, there exists a $d$-dimensional Brownian motion $B^x=((B^x)_t)_{t \geq 0}$ with respect to a standard extension of $(\Omega, (\mathcal{F}_t)_{t \geq 0}, P_x)$ such that $P_x$-a.~s.
	$$M^{[\pi_i]}_t = \myint{0}{\sqrt{2}\mathbbm{1}_{\R^d\setminus \SetA}(X_s)\,}{(B^x)_s^i}{t}$$
	for all $t \geq 0$ and for all $1 \leq i\leq d$. 
	In particular, it holds $M^{[\pi_i]}_{t \wedge \tau_k} = M^{[\pi_i^k]}_{t \wedge \tau_k}\ P_x$-a.s.. Forming the limit in Equation \eqref{eq: stopped process} yields the claim.
\end{proof}

\begin{thm}\label{thm: Process is Bm in 1d}
	Assume the same conditions as in Theorem \ref{thm: Fuk decomp in 1d}.
	For all $x \in \R$,  there exists a one-dimensional Brownian motion $B^x$ on a propably enlarged probability space $(\Omega, (\mathcal{F}_t)_{t \geq 0}, P_x)$ such that $P_x$-a.~s.
	\begin{align*}
		X_t- X_0 = \sqrt{2} \myint{0}{\mathbbm{1}_{\R\setminus\SetA}(X_s)\,}{({B^x})_s^i}{t} + \myint{0}{\mathbbm{1}_{\R\setminus\SetA}(X_s)(\ln \varrho)'(X_s)\,}{s}{t}
	\end{align*}
	for all $t \geq 0$. The last integrand has to be read as Lebesgue Integral when $\omega \in \Omega$ is fixed. We deviate from our usual notation to adapt to the common one.
\end{thm}
\begin{proof}
	The proof works as the proof of Theorem \ref{thm: Process is Bm}. Note that since we have strict additive funcitonals in Theorem \ref{thm: Fuk decomp in 1d}, we do not have to worry about exceptional sets and obtain the result for all $x \in \R$.
\end{proof}

\begin{rem}\label{rem: Bass process compared to our process}
	The theorems above show that the process $\Markov{\varrho}$ is a distorted Brownian motion outside of $\SetA$. 
	Furthermore, Theorem \ref{thm: Process is Bm in 1d} yields for the choice $A=\{0\}, \mS = \delta_0$ and $\varrho=\mathbbm{1}$ that $\Markov{\varrho}$ solves the SDE
	\begin{align*}
		d X_t = \sqrt{2}\mathbbm{1}_{\R\setminus\{0\}}(X_s)dB_t,\ X_0 =x,
	\end{align*}
	for $x \in \R$, where $(B_t)_{t \geq 0}$ is a one-dimensional standard Brownian motion. The proof of Theorem \ref{thm: Process is Bm in 1d} also shows that $(X_t)_{t \geq 0}$ is a continuous local martingale. Since the process has the sum of the Lebesgue measure and the Dirac measure at $0$ as speed measure, the uniqueness result in \cite{B2014}, Theorem 4.1 (p.6) implies that after multiplying the Dirichlet form with factor $\frac{1}{2}$ and repeating all arguments, our process and his solution coincide. \cite{B2014}, Remark 5.3 (p.12) is a sketch of a proof that his solution and the solution to the system of equations
	\begin{align*}\left\{
		\begin{array}{ll}
			dX_t &= \mathbbm{1}_{\{X_t\not=0\}}\, dB_t\\
			\mathbbm{1}_{\{X_t=0\}}\,dt &= \frac{1}{\mu}\, d\ell_t^0(X)
		\end{array}\right.,
	\end{align*}
	where $\mu\in(0,+\infty)$, $X:=\big(X_t\big)_{t\ge 0}$ and $\ell_t^0(X)$ is the local time in the semimartingale sense at $0$ of $X$,
	constructed in \cite{EP2014}, coincide. In particular this shows that in one dimension, using this result, that our process coincides with a process constructed by a time change of a Brownian motion.
\end{rem}

\subsection{Séjour Time and permeable behaviour}
In this section, we prove that the process $\Markov{\varrho}$ given in Theorem \ref{thm: Existence of processs}, which is associated to $(\Erho, D(\Erho))$ given as in \eqref{defn: Erho}, has positive séjour time (or sticky behaviour) on $\SetA$, \ie $\lim_{t \to \infty}\frac{1}{t}\myint{[0,t]}{\mathbbm{1}_{\SetA}(X_s)\,}{\lambda(s)}{}>0$ and that $\SetA$ is permeable if the measure $\murho$ is finite. To start with, we show that $(\Erho, D(\Erho))$ is irreducible.
\begin{prop}\label{prop: E is irreducible}
	Let $(\Erho, D(\Erho))$ be given as in \eqref{defn: Erho}. Assume that Condition \ref{cond: rho integrable} holds. Additionally, assume that Condition \ref{cond: d is one} or \ref{cond: Lipschitz boundary} is satisfied. Then $(\Erho, D(\Erho))$ is irreducible.
\end{prop}
\begin{proof}
	First let $d \in \N$ be arbitrary.
	Proposition \ref{prop: E is recurrent} shows that under Condition \ref{cond: rho integrable}, $(\Erho, D(\Erho))$ is recurrent. Condition \ref{cond: rho integrable} further assures $\murho(\R^d) < \infty$. The aim is to apply \cite{FukushimaChen}, Theorem 2.1.11 (p.46) to show the irreducibility of $(\Erho, D(\Erho))$.
	Let $f \in D(\Erho)$ with $\Erho(f,f)=0$. Then, $f^{(i)} =0\ \lambda^d$ \almosteverywhere on $\R^d$ for all $i=1,\ldots, d$ since $\varrho$ is positive. This implies that $f=c_f$ $\lambda^d$ \almosteverywhere on $\R^d$ for some $c_f \in \R$. It remains to check that $f=c_f$ $\mS$ \almosteverywhere on $\SetA$.  Under Condition \ref{cond: rho integrable} with the same arguments as in the proof of Proposition \ref{prop: Revuz correspondance and S00}(b), it holds that $\mathbbm{1} \in D(\Erho)$ and therefore, also $g \vcentcolon = f-c_f\mathbbm{1} \in D(\Erho) \subseteq \localSobolev{1,2}{\R^d}$.
	\par{}
	Assume that Condition \ref{cond: d is one} is satisfied. We find a sequence $(g_n)_{n \in \N} \subseteq D$ that converges to $g$ \withrespectto $(\Erho_1)^{\frac{1}{2}}$. In particular, we have
	\begin{align}\label{eq: E irreducible for $d=1$, Ltwomurho convergence}
		\norm{g_n-g}_{L^2(\R,\measdens{\mu}{\varrho})} \to 0\text{ as }n \to \infty.
	\end{align} 
	Choose $x_N \in \SetA$ arbitrary. For some $\varepsilon>0$, we have $B_{\varepsilon}(x_N) \subseteq (\R\setminus \SetA) \cup \{x_N\}$ and in particular $g_n, g \in \Sobolev{{1,2}}{B_{\varepsilon}(x_N)}$ with $\norm{g_n-g}_{\Sobolev{{1,2}}{B_{\varepsilon}(x_N)}}\to 0$. Since $\widetilde{g} \vcentcolon =0$ is the continuous representative of $g \in L^2(B_{\varepsilon}(x_N), \lambda)$, we also have $\sup_{x \in B_{\varepsilon}(x_N)} \abs{g_n(x)-\widetilde{g}(x)} \to 0$ as $n \to \infty$ by a Sobolev embedding theorem, see e.g.~\cite{AdamsSobSpaces}, Lemma 5.15 (p.107). This implies $g_n(x_N) \to 0$ as $n \to \infty$. On the other hand, $g_n(x_N) \to g(x_N)$ by Equation \eqref{eq: E irreducible for $d=1$, Ltwomurho convergence} and thus, $g(x_N)=0$. Therefore, since $x_N \in \SetA$ was arbitrary, $g=0\ \murho$ a.e..
	\par{}
	Assume that Condition \ref{cond: Lipschitz boundary} is satisfied.
	Since $\Lipschitz$ is bounded by assumption, $\restrictfunc{g}{\Lipschitz} \in \Sobolev{1,2}{\Lipschitz}$ with $\norm{\restrictfunc{g}{\Lipschitz}}_{\Sobolev{1,2}{\Lipschitz}}=0$. Using \cite{ALTLAEnglish}, A8.6 (p.268), we obtain the existence of an unique linear and continuous trace operator
	\begin{align*}
		T \vcentcolon \Sobolev{1,2}{\Lipschitz}\to L^2(\SetA, H^{d-1})
	\end{align*}
	with the property that $T u =u\big|_{\SetA}$ for all $u \in C^{\infty}(\overline{\Lipschitz})$. We conclude that
	\begin{align}\label{eq: irredcibility norm of T applied to g}
		\norm{T\left(g\big|_{\Lipschitz}\right)}_{L^2(\SetA, H^{d-1})} \leq \norm{T}\norm{g\big| _{\Lipschitz}}_{\Sobolev{1,2}{\Lipschitz}} =0.
	\end{align}
	Since $g \in D(\Erho)$ and $\varrho$ is strictly positive and continuous, there exists a sequence $(g_n)_{n \in \N} \subseteq D$ such that 
	\begin{align}
		\lim_{n \to \infty}\norm{\restrictfunc{g_n}{\Lipschitz}- \restrictfunc{g}{\Lipschitz}}_{L^2(\Lipschitz, \lambda^d)}= 0 \quad \text{ and }\quad
		\lim_{n \to \infty}\norm{\restrictfunc{g_n}{\SetA} - \restrictfunc{g}{\SetA}}_{L^2(\SetA, H^{d-1})}=0.\label{eq: irreducibility restriction of g on SetA}
	\end{align}
	Additionally, we have by Proposition \ref{prop: Domain is subset of Sob Space}
	\begin{align*}
		\lim_{n \to \infty}&\norm{\restrictfunc{\partial_i g_n}{\Lipschitz} - \restrictfunc{g^{(i)}}{\Lipschitz}}_{L^2(\Lipschitz, \lambda^d)} =0.
	\end{align*}
	This implies  $\lim_{n \to \infty}\norm{\restrictfunc{g_n}{\Lipschitz}- \restrictfunc{g}{\Lipschitz}}_{\Sobolev{1,2}{\Lipschitz}}= 0$ and thus 
	\begin{equation}\label{eq:  irreducibility t2 convergence on the boundary}
		\lim_{n \to \infty}\norm{T\left(\restrictfunc{g_n}{\Lipschitz}\right) - T \left(\restrictfunc{g}{\Lipschitz}\right)}_{L^2(\SetA, H^{d-1})}=0. 
	\end{equation}
	By construction, $(\restrictfunc{g_n}{\Lipschitz})_{n \in \N} \subseteq C^{\infty}(\overline{\Lipschitz})\cap \Sobolev{{1,2}}{U}$ and the definition of $T$ yields $T\left(\restrictfunc{g_n}{\Lipschitz}\right) = \restrictfunc{g_n}{\SetA}.$
	Plugging this in Equation \eqref{eq:  irreducibility t2 convergence on the boundary} we obtain
	\begin{equation}
		\lim_{n \to \infty}\norm{\restrictfunc{g_n}{\SetA} - T\left( \restrictfunc{g}{\Lipschitz}\right)}_{L^2(\SetA, H^{d-1})}=0
	\end{equation}
	and conclude together with Equation \eqref{eq: irreducibility restriction of g on SetA} and Equation \eqref{eq: irredcibility norm of T applied to g}  that $T \left(\restrictfunc{g}{\Lipschitz}\right) = \restrictfunc{g}{\SetA}=0$ in $L^2(\SetA, H^{d-1})$ and thus $f-c_f \mathbbm{1}=0$ $H^{d-1}$-\almosteverywhere on $\SetA$.
\end{proof}
\begin{thm}\label{thm: process is ergodic}
	Let $(\Erho, D(\Erho))$ be given as in \eqref{defn: Erho} and $\Markov{\varrho}$ be given as in Theorem \ref{thm: Existence of processs}. Assume that Condition \ref{cond: rho integrable} holds. Additionally, assume that Condition \ref{cond: d is one} or \ref{cond: Lipschitz boundary} is satisfied and let $f \in L^1(\R^d,\murho)$. It then holds
	\begin{align*}
		\lim_{t \to \infty}\frac{1}{t}\myint{[0,t]}{f(X_s)\,}{\lambda(s)}{} = \frac{\myint{\R^d}{f \, \varrho}{\mu}{}}{\murho(\R^d)}
	\end{align*}
	$P_x$-a.s.~for q.e.~$x \in \R^d$ if $d \geq 2$ and $P_x$-a.s.~for all $x \in \R$ if $d=1$.
\end{thm}
\begin{proof}
	Under these assumptions, $(\Erho, D(\Erho))$ is recurrent by Proposition \ref{prop: E is recurrent} and irreducible by Proposition \ref{prop: E is irreducible}. The claim then follows from \cite{Fukushima}, Theorem 4.7.3 (p.205). If $d=1$, then Lemma \ref{lem: cap zero}(a) implies that only the empty set has Capacity zero and therefore, the statement holds for all $x \in \R$.
\end{proof}

\begin{cor}\label{cor: positive sejour time}
	The choice of $f=\mathbbm{1}_{\SetA}$ in Theorem \ref{thm: process is ergodic} now yields the sticky behaviour of $\Markov{\varrho}$. 
\end{cor}
Note that the positive séjour time of the process on $\SetA$ comes from the fact that $\SetA$ is not a $\measdens{\mu}{\varrho}$ null set. 

We finally show the permeability of $\SetA$.
\begin{cor}\label{cor: permeable}
	Let $(\Erho, D(\Erho)$ be given as in \eqref{defn: Erho} and $\Markov{\varrho}$ be given as in Theorem \ref{thm: Existence of processs} and assume that Condition \ref{cond: rho integrable} holds.
	\begin{enumerate}[(i)]
		\item Let additionally Condition \ref{cond: d is one} be satisfied. For any $x_k \in \SetA$, let $\varepsilon_{x_k}>0$ be such that $(x_k-\varepsilon_{x_k}, x_k+\varepsilon_{x_k})\setminus\{x_k\}\cap \SetA=\emptyset$. Then for all $x \in \R$
		\begin{align*}
			&P_x\left(\inf\{t \geq 0 \vcentcolon X_{t+n} \in (x_k-\varepsilon_{x_k}, x_k)\}< \infty \text{ for all }n \geq 1\right)=1 \text{ and }\\
			&P_x\left(\inf\{t \geq 0 \vcentcolon X_{t+n} \in (x_k, x_k+\varepsilon_{x_k})\}< \infty \text{ for all }n \geq 1\right)=1 .
		\end{align*}
		\item If Condition \ref{cond: Lipschitz boundary} is satisfied, then for q.e.~$x \in \R^d$
		\begin{align*}
			&P_x\left(\inf\{t \geq 0 \vcentcolon X_{t+n} \in U\}< \infty \text{ for all }n \geq 1\right)=1 \text{ and }\\
			&P_x\left(\inf\left\{t \geq 0 \vcentcolon X_{t+n} \in \overline{U}^C\right\}< \infty \text{ for all }n \geq 1\right)=1 .
		\end{align*}
	\end{enumerate}
\end{cor}
\begin{proof}
	Proposition \ref{prop: E is recurrent} shows that $(\Erho, D(\Erho))$ is recurrent and Proposition \ref{prop: E is irreducible} states that $(\Erho, D(\Erho))$ is irreducible. The statement then follows directly from \cite{FukushimaChen}, Theorem 3.5.6(ii) (p.123), again using Lemma \ref{lem: cap zero}(a) to obtain the statement for all $x \in \R$ in part (a).
\end{proof}
\begin{rem}
	Corollary \ref{cor: permeable} implies that for q.e.~$x \in \R^d$ (if $d\geq 2$) and for all $x \in \R$ (if $d=1$), the set of paths crossing $\SetA$ infinitely often has probability one. This shows that  $\SetA$ is permeable for the process.
\end{rem}